\documentclass{article}[12pt]
\usepackage{amssymb}
\usepackage{amsmath,amsthm,amsfonts}
\usepackage[latin1]{inputenc}
\usepackage{graphicx}
\usepackage{hyperref}
\usepackage{enumerate}
\usepackage{tikz}
\usepackage{mathrsfs}
\usepackage{verbatim}
\usepackage{color,url}
\usepackage{cite}

\newtheorem{remark}{Remark}[section]

\newtheorem{lemma}[remark]{Lemma}
\newtheorem{theorem}[remark]{Theorem}
\newtheorem{proposition}[remark]{Proposition}
\newtheorem{corollary}[remark]{Corollary}

\newtheorem{problem}{Problem}

\newcommand{\cp}{\,\square\,}
\newcommand{\stp}{\,\boxtimes\,}

\newcommand{\msl}{\{\hspace*{-0.1cm}\{}
\newcommand{\msr}{\}\hspace*{-0.1cm}\}}
\newcommand{\adimms}{\operatorname{msad}}

\DeclareMathOperator{\ecc}{ecc}
\DeclareMathOperator{\diam}{diam}
\DeclareMathOperator{\edim}{edim}
\DeclareMathOperator{\mdim}{dim_{m}}
\DeclareMathOperator{\medim}{edim_{m}}
\DeclareMathOperator{\odim}{dim_{om}}
\DeclareMathOperator{\ldim}{dim_{lm}}
\DeclareMathOperator{\lodim}{dim_{lom}}
\DeclareMathOperator{\Amal}{Amal}

\begin{document}

\title{Multiset resolvability parameters in graphs: A survey with new results and open problems}

\author{Mohammad Farhan$^{a,}$\thanks{\texttt{mohammad.farhan@uca.es}}
\and Sandi Klav\v zar$^{b,c,d,}$\thanks{\texttt{sandi.klavzar@fmf.uni-lj.si}}
\and Dorota Kuziak$^{a,}$\thanks{\texttt{dorota.kuziak@uca.es}}
\and Ismael G. Yero$^{e,}$\thanks{\texttt{ismael.gonzalez@uca.es}}
}
\maketitle

\begin{center}
$^a$ Departamento de Estad\'istica e Investigaci\'on Operativa, Universidad de C\'adiz, Spain \\

\medskip
$^b$ Faculty of Mathematics and Physics, University of Ljubljana, Slovenia\\
\medskip
	
$^c$ Institute of Mathematics, Physics and Mechanics, Ljubljana, Slovenia\\
\medskip
	
$^d$ Faculty of Natural Sciences and Mathematics, University of Maribor, Slovenia\\
\medskip

$^e$ Departamento de Matem\'{a}ticas, Universidad de C\'adiz, Algeciras Campus, Spain \\
\medskip
\end{center}

\begin{abstract}
The metric dimension, which has lots of variants and numerous applications in other fields, is one of the most important and most extensively studied topics in metric graph theory. Results in which resolvability is achieved by considering multisets of distances from a fixed vertex, instead of vectors as in the original version, are surveyed. The concepts discussed are multiset dimension, outer multiset dimension, local multiset dimension, edge multiset dimension, and $k$-multiset antidimension. Along the way, sharp lower bounds on the outer multiset dimension of diameter two graphs and join graphs with edgeless graphs are proved, which solves two open problems from the literature. New results on graphs with local multiset dimension equal to two are also proved. In particular, such graphs are characterized among block graphs.  Finally, a list of open problems from the literature is compiled, and several new problems are added to the list for future research.
\end{abstract}

\noindent
{\bf Keywords:} graph distance; metric dimension; multiset dimension; outer multiset dimension; local multiset dimension; edge multiset dimension; $k$-multiset antidimension \\

\noindent
{\bf AMS Subj.\ Class.\ (2020)}: 05C12 

\section{Introduction}

The concepts of multiset resolving sets and multiset dimension find their roots in the classical resolving sets and metric dimension, respectively, which are classical parameters in graph theory that serve as some mathematical foundations for localization in networks. In the classical model, a network is equipped with a set of ``landmarks'' placed at specific points. Every other location in the network is then assigned a unique ``code'' based on its exact distance to each of these landmarks. Because the distances are recorded in a fixed, specific order, the set of such landmarks works like a positioning system, ensuring that no two locations share the same identity. This classical approach essentially defines the minimum number of landmarks needed to perfectly map out every single point in a structure without any ambiguity, which somehow mimics a GPS-like system. 

In such a traditional framework, a point is represented by ordered distances to a set of fixed landmarks. The multiset dimension reimagines this by removing the order of those distances. Instead of knowing exactly which distance belongs to which landmark, we only know the collection of distances involved. This shift creates a much more challenging environment for navigation and identification, as the loss of positional information requires a more unique arrangement of landmarks to ensure every point in the network remains distinguishable. Despite that this increases the difficulty of addressing the location abilities in graphs, it is valuable to consider such a version as it, in reward, also increases the capacity to model navigation and identification systems where sensors or landmarks are anonymous and cannot be uniquely distinguished from one another. By moving away from ordered coordinates, the multiset dimension serves as a more realistic ``metric worst-case scenario'' for ensuring system reliability and structural uniqueness in networks.

The metric dimension related parameters in graphs have been very attractive research topics in the last two decades, although the first notion regarding these parameters dates back to 1953 when a related concept was introduced for general metric spaces in~\cite{blumenthal-1953}. For the specific area of graph theory, the first information on these parameters came independently from the 1970's due to Slater \cite{slater-1975}, and Harary and Melter \cite{harary-1976}. However, one might notice that, indeed, the topic remained almost unstudied until the work \cite{Chartrand-2000} brought it back to attention. This topic attracted several investigations in various directions, including combinatorial, computational, and applied approaches. For example, a very remarkable application appeared in \cite{tillquist-2019}, where a sort of methodology was designed for embedding biological sequence data into Hamming graphs by means of a metric dimension related parameter. In addition, the obtained embedding was further used in machine learning algorithms that learn classifiers from such datasets. Since it is not our goal to deeply dig into the metric dimension of graphs, we suggest the interested reader to see the extra information on this concept, and other related ones, that can be found in the two surveys \cite{kuziak-2021,tillquist-2023}, and references cited therein.

To the best of our knowledge, the multiset dimension was introduced by Saenpholphat in 2009 in~\cite{saenpholphat-2009}. It is no surprise, however, that the paper went unnoticed for a long time, as it was published in a conference proceedings volume and in Thai language. Eight years later, the concept was independently rediscovered in~\cite{simanjuntak-2017} by Simanjuntak, Vetr\'{\i}k, and Mulia. The fact that the concept has been independently introduced (at least) twice demonstrates that it is natural and appealing to the wider community of metric graph theory.

The multiset dimension of graphs has received increasing attention in the last few years, and because of this, we aim in this work to collect the contributions that are known so far. Along the way, we also present some new contributions to the knowledge on this parameter, as well as list many open questions that appears to be of interest for consideration in future investigations.

\section{Preliminaries}

In this section, we first introduce three of the dimensions covered by this article and present basic relationships between them. Afterwards, standard graph theory definitions and related notation needed later on are presented. But first of all, along our exposition we shall assume $G = (V(G), E(G))$ is a connected graph, unless specifically stated the contrary. Also, given $u,v\in V(G)$, by $d_G(u,v)$ we denote the shortest-path distance in $G$ between $u$ and $v$. 

\subsection{Multiset dimension and two variants}

If $S =\{u_1,\dots,u_k\}$ is an ordered set of vertices of a connected graph $G$, then the \emph{metric representation} of a vertex $x\in V(G)$ with respect to $S$ is the vector 
$$r(x|S) = (d_G(x,u_1),\dots,d_G(x,u_k))\,.$$
The set $S$ is a {\em resolving set} for $G$ if all the vectors $r(x|S)$, $x\in V(G)$, are pairwise different. The \emph{metric dimension} $\dim(G)$ of $G$ is the cardinality of a smallest resolving set for $G$. If we require that the condition $r(x|S) \ne r(y|S)$ holds only for every pair $x,y$ of adjacent vertices of $G$, then $S$ is called a {\em local resolving set} of $G$. The minimum cardinality of a local resolving set is the {\em local metric dimension} ${\rm ldim}(G)$ of $G$~\cite{okamoto-2010}.

For $S$ and $x$ as above, the authors of~\cite{saenpholphat-2009, simanjuntak-2017} independently suggested to consider the {\em multiset representation of $x$ with respect to} $S$ defined as 
$$m_G(x|S)=\msl d_G(x, u_1), \ldots, d_G(x, u_k) \msr\,,$$
where $\msl \cdot \msr$ limits a multiset. The set $S$ is a \emph{multiset resolving set} for $G$ if all the multisets $m_G(x|S)$, $x\in V(G)$, are pairwise different. The \emph{multiset dimension} $\mdim(G)$ of $G$ is the cardinality of a smallest multiset resolving set. 

It was already noted in~\cite{saenpholphat-2009} that there exist graphs $G$ which admit no multiset resolving sets. In such a case, we set $\mdim(G) = \infty$. To avoid this problem, an ``outer'' version of multiset resolving sets was introduced as follows~\cite{gil-pons-2019}. A set $S\subseteq V(G)$ is an \emph{outer multiset resolving set} for $G$, if the multisets $m_G(x|S)$, $x\in V(G)\setminus S$, are pairwise different. The cardinality of a smallest outer multiset resolving set is the \emph{outer multiset dimension} of $G$, denoted by $\odim(G)$. Note that this structure avoids the problem of infiniteness. 

Yet another variant of the multiset dimension was introduced in~\cite{alfarisi-2019}. A set $S\subseteq V(G)$ is a \emph{local multiset resolving set} for $G$ if, for every $xy\in E(G)$, the multisets $m_G(x|S)$ and $m_G(y|S)$ are different. The cardinality of a smallest local multiset resolving set is the \emph{local multiset dimension} of $G$, denoted by $\ldim(G)$. 

A minimum size resolving set of a given graph $G$ is called a \emph{metric basis} of $G$. Analogously, a \emph{multiset basis} of $G$, an \emph{outer multiset basis} of $G$, and a \emph{local multiset basis} of $G$ are defined. We will also need the concept of a \textit{multiset distance irregular graph} which refers to a graph $G$ in which $m(x|V(G))\ne m(y|V(G))$ holds for any two vertices $x,y\in V(G)$.

To illustrate the concepts introduced, let us take a look at 
the $3$-cube $Q_3$ as shown in Fig.~\ref{fig:Q3}. As we will 
see later, bipartite graphs have a local multiset dimension 
equal to one, so $\ldim(Q_3) = 1$ (any vertex represents a local multiset basis). Further, by computer search, we obtained that $\mdim(Q_3) = \infty$ and that $\odim(Q_3) = 5$. 

\begin{figure}[ht]
\begin{center}
\begin{tikzpicture}[scale=0.7,style=thick]
\tikzstyle{every node}=[draw=none,fill=none]
\def\vr{4pt} 

\begin{scope}[yshift = 0cm, xshift = 0cm]
\path (0,0) coordinate (v1);
\path (3,0) coordinate (v2);
\path (1.6,1.2) coordinate (v3);
\path (4.5,1.2) coordinate (v4);
\path (0,3) coordinate (v5);
\path (3,3) coordinate (v6);
\path (1.6,4.2) coordinate (v7);
\path (4.5,4.2) coordinate (v8);
\draw (v1) -- (v2) -- (v4) -- (v8) -- (v6) -- (v5) -- (v7) -- (v3) -- (v1);
\draw (v3) -- (v4);
\draw (v1) -- (v5);
\draw (v6) -- (v2);
\draw (v7) -- (v8);
\draw (v1)  [fill=black] circle (\vr);
\draw (v2)  [fill=black] circle (\vr);
\draw (v3)  [fill=black] circle (\vr);
\draw (v4)  [fill=white] circle (\vr);
\draw (v5)  [fill=white] circle (\vr);
\draw (v6)  [fill=black] circle (\vr);
\draw (v7)  [fill=black] circle (\vr);
\draw (v8)  [fill=white] circle (\vr);
\draw[left] (v5)++(-0.1,0.0) node {$\msl 1^3, 2^2\msr$};
\draw[right] (v4)++(0.1,0.0) node {$\msl 1^2, 2^3\msr$};
\draw[right] (v8)++(0.1,0.0) node {$\msl 1^2, 2^2, 3\msr$};
\end{scope}

\end{tikzpicture}
\end{center}
\caption{An outer multiset basis $S$ of $Q_3$ is marked by the black vertices. Along with the remaining vertices $v$, the multiset $m_G(v|S)$ is written.}
\label{fig:Q3}
\end{figure}
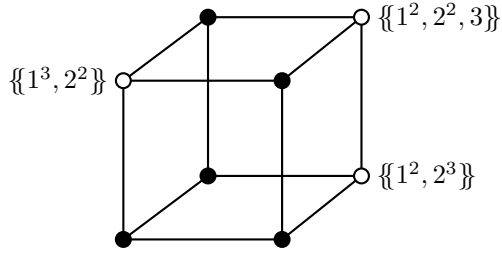

Now that, once we have introduced the dimensions of interest, we can immediately state that the following holds for any non-trivial graph $G$:
\begin{align*}
{\rm ldim}(G) & \le \ldim(G) \le \mdim(G)\,, \\
\dim (G) & \le \odim(G) \le \mdim(G)\,.
\end{align*}

\subsection{Other terminology and notation}

Let $G = (V(G), E(G))$ be a graph. The order of $G$ is denoted by $n(G)$. For $v\in V(G)$, the \emph{open neighborhood} $N_G(v)$ is the set of neighbors of $v$, while the \emph{closed neighborhood} $N_G[v]$ is the open neighborhood supplemented with the vertex $v$ itself. The maximum degree of $G$ is denoted by $\Delta(G)$. Two vertices, $u$ and $v$, are \emph{true twins} in $G$, if $N_G[u]=N_G[v]$, and they are \emph{false twins} if $N_G(u)=N_G(v)$. By saying that $u$ and $v$ are {\em twins} we mean that they are either false twins or true twins. A vertex $x$ is {\em simplicial} if $N_G[x]$ induces a complete graph. The {\em eccentricity} $\ecc_G(x)$ of a vertex $x\in V(G)$ is the maximum distance between $x$ and all the other vertices of $G$. The diameter $\diam(G)$ of $G$ is the maximum distance between all pairs of vertices of $G$, that is, $\diam(G) = \max_{x\in V(G)}\ecc_G(x)$. The clique number and the chromatic number of $G$ are denoted by $\omega(G)$ and $\chi(G)$, respectively. 

Standard graph classes like complete graphs, complete bipartite graphs, cycles, paths, and stars will be denoted respectively by $K_n$, $K_{r,s}$, $C_n$, $P_n$, and $S_n$. The \textit{complete $k$-ary tree of height $h$} is a rooted tree whose root has degree $k$, all its leaves are at a distance of exactly $h$ from the root, and its descendants are either leaves or vertices of degree $k+1$.

The join $G+H$ of disjoint graphs $G$ and $H$ is obtained from the disjoint union of $G$ and $H$ by adding an edge between every vertex in $G$ and every vertex in $H$. The complement of a graph $G$ is denoted by $\overline{G}$, while the subgraph of $G$ induced by a set $S\subseteq V(G)$ is written $G[S]$. The {\em Cartesian product} $G\cp H$ of graphs $G$ and $H$ as well as their lexicographic product $G\circ H$ and their strong product $G\boxtimes H$, all have the vertex set $V(G)\times V(H)$. Vertices $(g,h)$ and $(g',h')$ are adjacent in $G\cp H$ if either $gg'\in E(G)$ and $h=h'$, or $g=g'$ and $hh'\in E(G)$. These two vertices are adjacent in $G\circ H$ if either $g=g'$ and $hh'\in E(H)$, or $gg'\in E(G)$. And these two vertices are adjacent in $G\boxtimes H$ if either $gg'\in E(G)$ and $h=h'$, or $g=g'$ and $hh'\in E(G)$, or $gg'\in E(G)$ and $hh'\in E(G)$.

Finally, for a positive integer $k$, we will denote the set $\{1,\dots, k\}$ by $[k]$.

\section{Multiset dimension}
\label{sec:multiset}

In this section, we discuss the most fundamental concept of interest to us: the multiset dimension. In the first subsection we relate it to the so-called ID-colorings, continue by discussing the phenomenon of infinite multiset dimension, and complete by its computations complexity. The subsequent section collects bounds on the multiset dimension and specific exact results. After that we focus on the multiset dimension of trees, and of graph products. The section is concluded with the investigation of the multiset dimension on random graphs, digraphs, and graphs defined by algebraic structures. 

\subsection{ID-colorings, infinite multiset dimension, complexity}

As it frequently happens, the multiset dimension of graphs was also presented in a different approach by a group of authors other than the ones from \cite{saenpholphat-2009, simanjuntak-2017}. To see this, we refer to the work \cite{Chartrand-2021}, where the equivalent concept was introduced, and to \cite{hakanen-2024}, where such equivalence between the two concepts was noticed. 

Assume $G$ is a graph with $\diam(G) = d$, and let $S\subseteq V(G)$ be a set of vertices. For every vertex $x\in V(G)$, the \emph{code} of $x$ (with respect to the set $S$) is the $d$-vector $\vec{d}(x|S)=(t_1,\dots,t_d)$, where each $t_i$ represents the number of vertices in $S$ at distance $i$ from the vertex $x$. If the collection of codes of vertices of $G$ is pairwise different, then the set $S$ is called an \emph{identification coloring}, shortly an \emph{ID-coloring} of $G$. It must be remarked that there are graphs which have no ID-colorings. In this sense, a graph $G$ that has an ID-coloring is called an ID-graph. Now, for any ID-graph $G$, the cardinality of a smallest ID-coloring of $G$ is known as the \emph{ID-number} of $G$, which is also denoted by $ID(G)$. 

It was proved in \cite{hakanen-2024} that the ID-number and the multiset dimension of graphs represent the same thing. That is, the following result holds.

\begin{theorem}{\em\cite[Theorem 2.1 and Corollary 2.2]{hakanen-2024}}
\label{th:equivalence}
Let $G$ be a graph with $\diam(G) = d$. Then $S\subseteq V(G)$ is an ID-coloring for $G$ if and only if $S$ is a multiset resolving set for $G$. Moreover, $\mdim(G)=ID(G)$.
\end{theorem}

Other contributions to the multiset dimension of graphs, which are using the terminology of ID-colorings, are \cite{Cai-2025,Cueno-2025,Kono-2022,Kono-2022a,Kono-2021,Tolentino-2025}. Many of the known results which use the terminology of multiset resolving sets are also known in the  terminology of ID-colorings. Along this review, we shall use the terminology and notation from \cite{simanjuntak-2017}.

\medskip
As already mentioned, there are several graphs which have no multiset resolving sets and so, its multiset dimension is established as infinite. Such a fact was already noticed in~\cite{saenpholphat-2009}. For example, if a graph $G$ has three vertices that are pairwise true twins (or pairwise false twins), then its multiset dimension is infinite~\cite[Theorem 3]{saenpholphat-2009}. In addition, the next result is an interesting one in this regard.

\begin{theorem}{\em \cite[Theorem 2]{simanjuntak-2017}}
\label{th:diam-2}
If $G$ is a graph of diameter two different from a path, then $\mdim(G)=\infty$.
\end{theorem}

After such a contribution, the problem of characterizing all the graphs $G$ satisfying $\mdim(G)=\infty$ became an interesting one. However, due to the wide structure of such graphs, it seems to be a very challenging question to settle such a problem. For example, it can be noticed that the diameter of the graph is not a restriction for it to have infinite multiset dimension. Recall the example of the hypercube $Q_3$, which has diameter three and infinite multiset dimension. This might come from the very high symmetry that $Q_3$ has, and it suggests that perhaps any hypercube satisfies a similar situation. However, as we will further see in Section \ref{sec:conclude}, this is not the case, at least for a few other smaller hypercubes. The problem of characterizing the graphs with infinite multiset dimension has also been approached from some other directions. For example, in~\cite{Cevik-2024,Cevik-2025}, the authors claimed to have considered this problem via the graphs obtained from special minimal (while inefficient or not) monoid presentations, and to have investigated some ideas about matching special graph dimensions, like the (local) multiset dimension, in terms of ideals over non-commutative rings, or more generally, on algebras. The topics of these works are somehow out of the scope of this survey.

\medskip
The problem of computing the multiset dimension was proved to be NP-hard in \cite{hakanen-2024}. There, the authors used a reduction from the 3-SAT problem which was highly influenced by a similar result from \cite{gil-pons-2019} showing the NP-hardness of computing the outer multiset dimension of graphs. In addition, such a reduction is also influenced by the classical reduction, also from 3-SAT, for showing the NP-hardness of the classical metric dimension. The decision problem regarding the multiset dimension is as follows.

\bigskip
\indent \textsc{Multiset Dimension Problem} \\
\indent \textbf{Instance:} A graph $G$, an integer $k\in [n(G)]$. \\
\indent \textbf{Question:} Is there a multiset resolving set in $G$ of cardinality at most $k$?
\bigskip

\noindent
The following result from \cite{hakanen-2024} leads then to the NP-hardness previously mentioned.

\begin{theorem}{\em \cite[Theorem 3.1]{hakanen-2024}}
The {\sc Multiset Dimension Problem} is NP-complete.
\end{theorem}

\subsection{Bounds and exact results}

A natural question that is usually considered for any graph parameter relates to its realization. This issue was presented in \cite{Khemmani-2018}, where it was specifically proved the following.

\begin{theorem}{\em \cite[Theorem 5]{Khemmani-2018}}
For every pair $k, n$ of integers with $k\ge 3$ and $n \ge 3(k - 1)$, there exists a connected graph $G$ of order $n$ with $\mdim(G) = k$.
\end{theorem}

Any graph $G$ that admits a multiset resolving set satisfies that 
$1\le \mdim(G)\le n(G)\,.$ 
The graphs that attain the lower bound above follow the steps of other related metric parameters.

\begin{theorem}{\em \cite[Theorem 7]{saenpholphat-2009}}
Let $G$ be a graph. Then $\mdim(G)=1$ if and only if $G$ is a path.
\end{theorem}

In this sense, it should happen that any graph different from a path will have multiset dimension at least two. However, as noted in~\cite[Theorem 8]{saenpholphat-2009}, there is no graph with multiset dimension two, which is based on the simple fact that for any two vertices $x,y$, it holds that $m_G(x|\{x,y\})=\msl 0,d_G(x,y)\msr=m_G(y|\{x,y\})$.

\begin{theorem} {\rm \cite[Theorem 8]{saenpholphat-2009}}
If $G$ is not a path, then $\mdim(G)\ge 3$.
\end{theorem}

Concerning this fact, it is then natural to think about characterizing those graphs whose multiset dimension equals three. For instance, as proved in~\cite[Theorem 9]{saenpholphat-2009}, we have $\mdim(C_n) = 3$ for $n\ge 6$. For graphs with multiset dimension $3$, the following contribution was given in \cite{bong-2021}. 

\begin{theorem}{\em \cite[Theorem 2.3]{bong-2021}}
If $G$ is a graph with $\diam(G) = d$ and $\mdim(G) = 3$, then $n(G)\le \frac{d^3+3d^2+2d+12}{6}$.
\end{theorem}

The proof of the result above strongly relies on the following property that a graph $G$ with $\mdim(G)=3$ has. If $S=\{x,y,z\}$ is a multiset basis of $G$, then the three distances $d_G(x,y)$, $d_G(x,z)$, and $d_G(y,z)$ are pairwise different. Otherwise, if for instance $d_G(x,y)=d_G(x,z)$, then clearly $m_G(y|S)=\msl0,d_G(y,x),d_G(y,z)\msr=\msl0,d_G(z,x),d_G(z,y)\msr=m_G(z|S)$, which is not possible.

\smallskip
On the other hand, regarding the trivial upper bound $\mdim(G)\le n(G)$, it is not clear whether there is a graph $G$ for which $\mdim(G)=n(G)$. It was conjectured in \cite[Conjecture 1]{hfidh-2019} that if $G$ admits a multiset resolving set, then $\mdim(G)\le n(G)-1$. 

Several other minor results on the multiset dimension were given. Using the terminology of ID-colorings, exact results on some families of graphs were given in~\cite{Cai-2025}. In particular, a study on lollipop graphs was presented, where a {\em lollipop graph} is obtained from a cycle and a path by identifying one leaf of the path with any vertex of the cycle. A family of chemical graphs named starphene graphs was studied in~\cite{liu-2021} and proved that the multiset dimension for every graph in this family is equal to $4$. In addition, the work~\cite{Bengeri-2025} presents some studies on the multiset dimension of those graphs obtained from two cycles and a path, by joining one vertex of one cycle with a vertex of degree one in the path and any vertex from the other cycle with the other vertex of degree one from the path. 

\subsection{Multiset dimension of trees}

To determine the multiset dimension of an arbitrary tree appears to be a very challenging problem. In this subsection we collect known results in this direction. We begin with the following theorem which supports~\cite[Conjecture 1]{hfidh-2019} asserting that $\mdim(G)\le n(G)-1$, assuming that $G$ admits at least one multiset resolving set. 

\begin{theorem}{\em \cite[Theorem 1]{hfidh-2019}}
If $T$ is a tree such that $\mdim(T)<\infty$, then $\mdim(T)\le n(T)-2$. 
\end{theorem}

Below we present known results for various tree families. The first part of the next results follows by the fact that if a graph $G$ contains a vertex which is adjacent to at least three leaves, then $\mdim(G) = \infty$. 

\begin{theorem}{\em \cite[Theorem 5]{simanjuntak-2017}}
Let $G$ be a complete $k$-ary tree of height $h$, where $k\ge 2$ and $h\ge 1$. Then $\mdim(G)=\infty$ if and only if $k\ge 3$. Moreover, if $k=2$, then $\mdim(G)=2^h-1$.
\end{theorem}

Given a graph $G$, a subgraph $P$ of $G$ is called a $k$-\textit{center path} of $G$, if $P$ is a path and $d_G(u,V(P))=\min\{d_G(u,v)\,:\,v\in V(P)\}\le k$ for every $u\in V(G)$. A minimum $k$-center path is a $k$-center path of the smallest possible length. A \textit{caterpillar} is a tree containing a 1-center path, and a \textit{lobster} is a tree containing a 2-center path. 

\begin{theorem}{\em \cite[Theorem 2]{hfidh-2019}}
Let $G$ be a lobster. If $P$ is the minimum $2$-center path of $G$, then the following claims are equivalent.
\begin{itemize}
    \item[(i)] $G$ has finite multiset dimension.
    \item[(ii)] The only component of $G-E(P)$ with infinite multiset dimension is an $S_4$.
    \item[(iii)] If $H$ is a $G-E(P)$, then $G[H]$ has at most four components which are either a $P_2$, a $P_3$, or an $S_4$, with at most two $P_2$ and two $S_4$.
\end{itemize}
\end{theorem}

\begin{theorem}{\em \cite[Theorem 3]{hfidh-2019}}
Let $G$ be a caterpillar with $P$ as its minimum $1$-center path. Then $\mdim(G)<\infty$ if and only if every vertex in $V(P)$ has at most two neighbors in $G-P$.
\end{theorem}

Similar and indeed equivalent results to the ones above were obtained in \cite{Kono-2022} by using the terminology of ID-coloring. In addition, such work gives the following contributions.

\begin{theorem}{\em \cite[Theorem 2.1]{Kono-2022}}
If $T$ is a caterpillar with $\diam(T)\ge 3$ and $\Delta(T) = 3$, then $\mdim(T)=3$.
\end{theorem}

\begin{theorem}{\em \cite[Theorem 2.2]{Kono-2022}}
Let $T$ be a caterpillar with $\diam(T)\ge 4$ and $\Delta(T) = 4$. If $T$ has $k$ vertices of degree at least three adjacent to two leaves, then $\mdim(T)<\infty$. Moreover, $\max\{3,k\}\le \mdim(T)\le k+3$.
\end{theorem}

In connection with the result above, a realization result for caterpillars with prescribed values in the multiset dimension was also provided. That is, they proved that every integer between $\max\{3, k\}$ and $k+3$ is realizable as the multiset dimension of some caterpillar. In addition, regarding the multiset dimension of caterpillars, a very technical characterization of the caterpillar graphs with a multiset dimension equal to $3$ was presented in \cite{Khemmani-2020}. Moreover, several other contributions on caterpillar graphs (with special emphasis on the symmetric ones) were given in \cite{Isariyapalakul-2020}.

\subsection{Multiset dimension of graph products}

Here we primarily present some contributions regarding the multiset dimension of Cartesian products, and we also add a couple of results on the strong product. For the grids, that is, Cartesian products of paths, we have the following. 

\begin{theorem}
{\em \cite[Theorem 7]{simanjuntak-2017}} For any integers $r,s\ge 2$ such that $(r,s)\ne (2,2)$, it holds $\mdim(P_r\cp P_s) = 3$.
\end{theorem}

In \cite{marcelo-2025}, cylindrical graphs were studied, namely, those graphs obtained as the Cartesian product of a path and a cycle. The main results read as follows, while the multiset dimension of the remaining cases remains an open problem. 

\begin{theorem}{\em \cite[Theorems 3.4 and 3.5]{marcelo-2025}}
\begin{itemize}
    \item[(i)] If $m\ge 3$, then $\mdim(C_3\cp P_m)=3$.
    \item[(ii)] If $m\ge 8n+1$, then $\mdim(C_n\cp P_m)=3$.
\end{itemize}
\end{theorem}

Cartesian products of an arbitrary graph with a complete graph have also been investigated. Recalling that in a multiset distance irregular graph $G$, by definition, $m(x|V(G))\ne m(y|V(G))$ for any $x,y\in V(G)$, we have the following results.  

\begin{theorem}{\em \cite[Theorem 5.1]{hakanen-2024}}
Let $G$ be a graph.
\begin{enumerate}
    \item[(i)] If $n\ge 3$ is an integer,  then $\mdim(G\cp K_n)=\infty$.
    \item[(ii)] If $G$ is not a multiset distance irregular graph, then $\mdim(G\cp K_2)=\infty$.
\end{enumerate}
\end{theorem}

\begin{theorem}{\em \cite[Corollary 5.1]{hakanen-2024}}
If $G$ is multiset distance irregular, then $\mdim(G\cp K_2) = n(G)$.
\end{theorem}

Strong products of paths were considered in~\cite{hakanen-2024}. Because $\diam(P_2\stp P_2) = 1$ and since $\diam(P_3\stp P_3) = 2$, these two products have no multiset resolving sets, cf.~Theorem~\ref{th:diam-2}. Further, $\mdim(P_4\stp P_4)=6$ and $\mdim(P_5\stp P_5)=\mdim(P_6\stp P_6)=4$, while for all the other dimension in the product we have: 

\begin{theorem}{\em \cite[Theorem 4.3]{hakanen-2024}}
If $n\ge 7$, then $\mdim(P_n\stp P_n) \in \{3,4\}$.
\end{theorem}

\subsection{Other directions}

In this subsection we cover other directions of the investigation of the multiset dimension: random graphs, digraphs, and graphs defined by algebraic structures. 

Eide and Pra{\l}at~\cite{eide-2026+} investigated the mutiset dimension on random graphs. To formulate their main result, some preparation is needed. An event holds with high probability (w.h.p.) if it holds with probability tending to one as $n \rightarrow \infty$. The binomial random graph ${\cal G}(n, p)$ is defined as a distribution over the class of graphs with the vertex set $[n]$ in which every pair $\{i, j\} \in \binom{[n]}{2}$ appears independently as an edge in $G$ with probability $p$. For $x \in (0, 1]$, the function $f_x$ is defined on $[0, 1]$  by
$$f_x(y) = \sum_{i=0}^{\left\lfloor \frac{1}{x}\right\rfloor} \max \{ix + y - 1, 0\}\,.$$
The main result of Eide and Pra{\l}at then reads as follows. 

\begin{theorem} {\rm \cite[Theorem 1.4]{eide-2026+}}
\label{thm:pralat}
If $x \in \left(0, \frac{1}{8}\right]$ and $d = (n - 1)p = n^{x+O(\log^{-1} n)}$, then w.h.p.
$$\mdim({\cal G}(n, p)) \le n^{y_4 + O(\log^{-1} n)}\,,$$
where $y_4 = y_4(x) \in (0, 1)$ is the unique solution to $f_x(y) = 4$. On the other hand, if $x \in \left(0, \frac{1}{2}\right]$ and $d = n^{x+O(\log^{-1} n)}$, then w.h.p.
$$\mdim({\cal G}(n, p)) \ge  n^{y_1 + O(\log^{-1} n)}\,,$$
where $y_1 = y_1(x)$ is the unique solution to $f_x(y) = 1$. Finally, if $x > \frac{1}{2}$, then w.h.p.
$$\mdim({\cal G}(n, p)) = \infty\,.$$
\end{theorem}

The concept of multiset dimension naturally extends to digraphs by simply considering the distance function in digraphs. This was done for the first time by Cueno, Garciano, and Marcelo in~\cite{Cueno-2025} by considering the multiset dimension of Cayley digraphs of Abelian groups. More precisely, they have focused on Cayley graphs ${\rm Cay}(\mathbb{Z}_n, \{a_1, \dots, a_t\})$, where $n$, $t$, and $a_1, \dots, a_n$ are positive integers such that $1\le a_1 < \cdots < a_t \le \lfloor n/2 \rfloor$. These Cayley graphs are also known as the circulant digraphs $C_n(a_1, \dots, a_t)$. They have determined $\mdim(C_n(1, t))$ for $t\ge 2$ and $n \ge 2t^2$ and proved an upper
bound for $\mdim(C_n(1, \dots, t))$ for $t \ge 3$ and $n \ge 2$. Moreover, they have established the following two results. 

\begin{theorem} {\rm \cite[Theorem 21]{Cueno-2025}}
If $n \ge 6$ and $H$ is the Cayley digraph of $\mathbb{Z}_n \oplus \mathbb{Z}_n$ with the generating set $\{(1,0), (0,1)\}$, then $\mdim(H) = n$.  
\end{theorem}

\begin{theorem} {\rm \cite[Theorem 22]{Cueno-2025}}
If $n \ge 3$, $m\ge n+1$, and $H$ is the Cayley digraph of $\mathbb{Z}_m \oplus \mathbb{Z}_n$ with the generating set $\{(1,0), (0,1)\}$, then $\mdim(H) = n$.  
\end{theorem}

Another research direction concerning the multiset dimension of graphs focuses on studying some graphs arising from some algebraic structures. Examples of this type of study are \cite{ali-2024a,ali-2024b,ali-2024c,ali-2026}. In general, the goals of these investigations are centered into taking some algebraic structure (with some particular interest) and generate a graph that somehow ``captures'' some properties of it, and then study the multiset dimension (or indeed any other parameter) of the graph constructed in this way, so that later on, some knowledge can be given about this algebraic structure, based on the graph theoretical property that has been considered. One of the problems that have such contributions is that very frequently the constructed graphs are rather simple and lack of much interest. Thus, the main contributions of these investigations are not precisely the graph theoretical conclusions, but the implications they might have in the algebraic structure. In this sense, in Table \ref{tab:placeholder} we simply comment in our expository work about the specific algebraic structures that have been considered so far.

\begin{table}[ht]
    \centering
    \begin{tabular}{|c|c|}
    \hline
       \cite{ali-2024a}  & compressed zero-divisor graphs over rings  \\ \hline
       \cite{ali-2024b}  & fuzzy zero divisor graph over a commutative ring \\ \hline
       \cite{ali-2024c}  & zero divisor graphs associated with rings\\ \hline
       \cite{ali-2026}   & $V$-graphs over commutative rings:  \\ 
        &  a graph obtained from the non-zero elements of a commutative ring\\ \hline
    \end{tabular}
    \label{tab:placeholder}
\end{table}

In order to see the structure of the results that appear in these types of investigation, we note the following one, where we need an extra definition. Let $S$ be a commutative ring with identity. The $V$-graph of $S$, denoted by $V(S)$, is the undirected graph whose vertex set is $S\setminus\{0\}$, that is, the vertex set of $V(S)$ is formed by all the non-zero elements of $S$. Two distinct vertices $x$ and $y$ of $V(S)$ are adjacent in $V(S)$ if the product $xy$ is a regular element of $S$.

\begin{proposition}{\em \cite[Proposition 3.1.1.]{ali-2026}} 
Let $R$ be a commutative ring with identity. Then the multiset dimension of the $V$-graph $V(R)$ is $1$ if and only if $R\cong \mathbb{Z}_3$.
\end{proposition}

We end this section by considering some related contributions regarding multiset resolving sets of graphs presented by using the ID-coloring terminology. In \cite{Chartrand-2021}, those graphs $G$ satisfying that $\mdim(G)<\infty$ were called \textit{ID-graphs}. Moreover, they defined the \textit{identification spectrum} of $G$ as the set of integers $r$ such that $G$ has a multiset resolving set of cardinality $r$. Clearly, this latter definition makes sense because a graph could not have a multiset resolving set of a specific cardinality $r'$, even if $\mdim(G)< r' \le n(G)$. Notice that if $S\subset V(G)$, then it might happen that $S\cup S'$ is not a multiset resolving set for every or for some $S'\subseteq V(G)\setminus S$. In addition, we may observe that the smallest integer belonging to the identification spectrum of a graph $G$ is, in fact, equal to $\mdim(G)$. 

With respect to this identification spectrum, the following results were also given in \cite{Chartrand-2021}.

\begin{theorem}{\em \cite[Theorem 3.1]{Chartrand-2021}}
For any integer $n\ge 4$, the path $P_n$ has a multiset resolving set of cardinality $r$ if and only if $r=1$ or $3\le r\le n-1$.
\end{theorem}

\begin{theorem}{\em \cite[Theorem 4.4]{Chartrand-2021}}
For any integer $n\ge 6$, the cycle $C_n$ has a multiset resolving set of cardinality $r$ if and only if $3\le r\le n-3$.
\end{theorem}

\section{Outer multiset dimension}
\label{sec:outer-multiset}

In the first subsection of this section, we summarize the known results regarding the outer multiset dimension. In the subsequent subsection, we prove two new lower bounds on the outer multiset dimension, thus solving three problems from the literature. 

\subsection{Summary of known results}

As a starting example of this section, consider the cycle graphs. It can be easily checked that $\odim(C_3) = 2$, $\odim(C_4) = 3$, and $\odim(C_5) = 4$, while for longer cycles the following holds. 

\begin{proposition} {\rm \cite[Proposition 3.5]{gil-pons-2019}}
\label{prop:odim-cycles}
If $n\ge 6$, then $\odim(C_n) = 3$.
\end{proposition}

A useful property of outer multiset resolving sets is the next presented. 

\begin{proposition}  {\rm \cite[Proposition 3.7]{gil-pons-2019}}
\label{prop:outer-twins}
If $S$ is an outer multiset resolving set of a graph $G$ and $u$ and $v$ are twins in $G$, then $u\in S$ or $v\in S$.
\end{proposition}

We now right away consider the computational complexity of finding $\odim$. To this end, the corresponding decision problem is: 

\bigskip
\indent \textsc{Outer Multiset Dimension Problem} \\
\indent \textbf{Instance:} A graph $G$, an integer $k\in [n(G)-1]$. \\
\indent \textbf{Question:} Is there an outer multiset resolving set in $G$ of cardinality at most $k$?
\bigskip

\noindent

\begin{theorem} {\rm \cite[Theorem 4.1]{gil-pons-2019}}
\label{thm:odim-np-complete}
The {\sc Outer Multiset Dimension Problem} is NP-complete.
\end{theorem}

With such a result, we conclude that the {\sc Outer Multiset Dimension Problem} is a tough one. It seems that the problem could still be tough even for trees, as indicated by the investigation of the outer multiset dimension of complete $\delta$-ary trees from~\cite{gil-pons-2019}. There, a formula for $\odim$ of these trees is given, but it leads to an exponential algorithm to determine the actual value. On the positive side, an explicit formula is given for complete 2-ary trees in~\cite[Corollary 5.4]{gil-pons-2019}. 

The discussion above suggests that it makes sense to explore general bounds for the outer multiset dimension as well as its behavior on specific important graph classes. The dimension of any non-trivial connected graph $G$ is bounded as follows (see \cite[Proposition 3.1]{gil-pons-2019}):  
$$1\le \odim(G) \le n(G) -1\,.$$
The lower bound is obvious, and the upper bound holds because $V(G)\setminus \{u\}$ is an outer multiset resolving set for every vertex $u$ of $G$. Graphs that attain the two bounds are characterized as follows. 

\begin{theorem} \label{thm:outer-extremal-dims} 
If $G$ is a connected graph, then the following holds.
\begin{enumerate}
\item[(i)] {\rm \cite[Proposition 3.2]{gil-pons-2019}} $\odim(G) = 1$ if and only if $G$ is a path graph.
\item[(ii)] {\rm \cite[Theorem 2.1]{klavzar-2023}} $\odim(G) = n(G) - 1$ if and only if $G$ is regular and $\diam(G) \le 2$.
\end{enumerate}
\end{theorem}

Graphs $G$ with $\odim(G) = 2$ were studied in detail in~\cite{klavzar-2023}. We note in passing that the problem of characterizing the graphs $G$ having $\dim(G) = 2$ is one of the open problems in the area, see~\cite{behtoei-2017} for a partial contribution in this direction. It is not difficult to establish that if $G$ is a graph with $\odim(G) = 2$ and $S=\{u,v\}$ is an outer multiset basis, then $d_G(u,v)\le 2$, see~\cite[Lemma 3.1]{klavzar-2023}. This fact leads to:

\begin{theorem} {\rm \cite[Theorem 3.2]{klavzar-2023}}
\label{thm:outer-dim=2}
Deciding whether a graph $G$ satisfies $\odim(G) = 2$ can be done in $\mathcal{O}(n(G)^3)$ time.
\end{theorem}

The structure of graphs $G$ with $\odim(G) = 2$ having an outer multiset basis formed by two adjacent vertices can be explicitly described as follows. Let $\mathcal{F}$ be the family of graphs $G$ constructed in the following way. Let $V(G)=\{x_0,\dots,x_r\}\cup \{y_0,\dots,y_s\}$, $r\ge 0$, $s\ge 1$. The edges of $G$ are given as described below.
\begin{itemize}
  \item $x_0y_0,x_0y_1\in E(G)$.
  \item For every $i\in [r]$ and every $j\in [s]$, $x_{i-1}x_{i}\in E(G)$ and $y_{j-1}y_{j}\in E(G)$.
  \item For every $i\in [\min\{r,s\}]$, the edge $x_iy_i$ might exist or not in $G$.
  \item For every $j\in [\min\{r,s-1\}]$, the edge $x_iy_{i+1}$ might exist or not in $G$.
\end{itemize}
For $r=0$ and $s=1$ we get $K_3\in \mathcal{F}$, and for $r=0$ and $s=2$ we get that the graph obtained from $K_3$ by attaching a pendant edge to one of its vertices also belongs to $\mathcal{F}$. The case $r=s=1$ yields $K_4-e \in \mathcal{F}$. In general, a representative example of a graph from $\mathcal{F}$ can be seen in Fig.~\ref{fig:dim-2-adjacent} with its outer multiset basis formed by the two black vertices. 

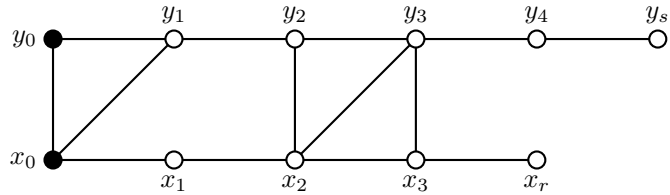
\begin{figure}[ht!]
\begin{center}
\begin{tikzpicture}[scale=0.8,style=thick]
\tikzstyle{every node}=[draw=none,fill=none]
\def\vr{4pt} 
\begin{scope}[yshift = 0cm, xshift = 0cm]
\path (0,0) coordinate (v1);
\path (2,0) coordinate (v2);
\path (4,0) coordinate (v3);
\path (6,0) coordinate (v4);
\path (8,0) coordinate (v5);
\path (0,2) coordinate (v6);
\path (2,2) coordinate (v7);
\path (4,2) coordinate (v8);
\path (6,2) coordinate (v9);
\path (8,2) coordinate (v10);
\path (10,2) coordinate (v11);
\draw (v5) -- (v1) -- (v6) -- (v11);
\draw (v1) -- (v7);
\draw (v8)-- (v3) -- (v9) -- (v4);
\draw (v1)  [fill=black] circle (\vr);
\draw (v2)  [fill=white] circle (\vr);
\draw (v3)  [fill=white] circle (\vr);
\draw (v4)  [fill=white] circle (\vr);
\draw (v5)  [fill=white] circle (\vr);
\draw (v6)  [fill=black] circle (\vr);
\draw (v7)  [fill=white] circle (\vr);
\draw (v8)  [fill=white] circle (\vr);
\draw (v9)  [fill=white] circle (\vr);
\draw (v10)  [fill=white] circle (\vr);
\draw (v11)  [fill=white] circle (\vr);
\draw[left] (v1)++(-0.1,0.0) node {$x_0$};
\draw[below] (v2)++(0.0,-0.1) node {$x_1$};
\draw[below] (v3)++(0.0,-0.1) node {$x_2$};
\draw[below] (v4)++(0.0,-0.1) node {$x_3$};
\draw[below] (v5)++(0.0,-0.1) node {$x_r$};
\draw[left] (v6)++(-0.1,0.0) node {$y_0$};
\draw[above] (v7)++(0.0,0.1) node {$y_1$};
\draw[above] (v8)++(0.0,0.1) node {$y_2$};
\draw[above] (v9)++(0.0,0.1) node {$y_3$};
\draw[above] (v10)++(0.0,0.1) node {$y_4$};
\draw[above] (v11)++(0.0,0.1) node {$y_s$};
\end{scope}
\end{tikzpicture}
\end{center}
\vspace*{-0.4cm}
\caption{A graph from the family $\mathcal{F}$, where $r=4$ and $s=5$.}
\label{fig:dim-2-adjacent}
\end{figure}

\begin{theorem} {\rm \cite[Theorem 3.4]{klavzar-2023}}
\label{th:multiset-dim-2-adj}
A graph $G$ has an outer multiset basis formed by two adjacent vertices if and only if $G\in \mathcal{F}$.
\end{theorem}

As already mentioned, for any non-trivial graph $G$, we have $\odim(G) \ge \dim(G)$. By considering, for a given set of vertices $S$ of $G$, the number of possible different multiset representations for the elements of $V(G)\setminus S$, the following strict inequality can be deduced. 

\begin{theorem} {\rm \cite[Theorem 3.6]{gil-pons-2019}}
\label{thm:outer-strict-inuquality-by-counting} 
If $G$ is a connected graph with $\dim(G) < k$, where $k$ is the smallest positive integer such that $k + \binom{n(G) + \diam(G) - 1}{\diam(G) - 1} \ge n(G)$, then $\odim(G) > \dim(G)$.
\end{theorem}

The outer multiset dimension has also been investigated on graph products. To state the next result for the lexicographic product, recall once more that in a multiset distance irregular graph $G$, for every two vertices $u$ and $v$, the multisets $m_G(u|V(G))$ and $m_G(v|V(G))$ are different.

\begin{theorem} {\rm \cite[Theorem 4.1]{klavzar-2023}}
\label{thm:lex}
If $G$ is a graph with $n(G)\ge 2$ and $H\in \{K_k, \overline{K_k}\}$, $k\ge 2$, then
$$\odim(G\circ H) \ge n(G)(k-1)\,.$$
Moreover, equality holds if and only if $G$ is multiset distance irregular.
\end{theorem}

For the grid graphs, that is, the Cartesian product of two paths, the following holds. 

\begin{theorem} {\rm \cite[Theorem 5.1]{klavzar-2023}}
\label{thm:grid}
If $s\ge t\ge 2$, then $\odim(P_s \cp P_t) = 3$.
\end{theorem}

In~\cite[Problem 6.1]{klavzar-2023}, the problem of investigating the outer multiset dimension of non-regular graphs of diameter $2$ was posed. This problem was first addressed in~\cite{pervaiz-2025} by considering the join of two graphs. Such graphs are of diameter (at most) two and are, in many cases, non-regular. The joins considered in this paper are stars $S_n = \overline{K_{n-1}} + K_1$, wheels $W_n = C_{n-1} + K_1$, generalized wheels $W_{m,n} = C_n + \overline{K_m}$, windmills $W(k, n) = nK_k + K_1$, fans $F_n = P_{n-1} + K_1$, and generalized fans $F_{m,n} = P_n + \overline{K_m}$. For each of these graphs, the outer multiset dimension is established. For instance, if $n\ge 3$, then $\odim(S_n) = n - 2$ by~\cite[Theorem 2.1]{pervaiz-2025}. From the other formulas, we extract the following two.

\begin{theorem} {\rm \cite[Theorem 2.4]{pervaiz-2025}}
\label{thm:fan}
If $n\ge 4$, then
$$
\odim(F_n) =\left\{\begin{array}{ll}
       2; 			& n \in \{4,5\}, \\
       3; 			& n = 6, \\
       n-4; 	& n \ge 7.                      
                    \end{array}
\right.
$$
\end{theorem}

\begin{theorem} {\rm \cite[Theorem 2.5]{pervaiz-2025}}
\label{thm:generalized-fan}
If $n\ge 2$ and $m\ge n$, then
$$
\odim(F_{m,n} = P_n + \overline{K_m}) =\left\{\begin{array}{ll}
       m; 			& n \in \{2,3\}, \\
       m+1; 			& n = 4, \\
       m+n-4; 	& n \ge 5.                      
                    \end{array}
\right.
$$
\end{theorem}

To conclude this subsection, we notice that the outer multiset dimension has been studied in a couple of papers on certain graphs arising from rings. In~\cite{riaz-2025}, the outer multiset dimension has been studied on zero-divisor graphs, and in~\cite{ali-2026}, the outer multiset dimension has been considered on the V-graphs over rings.

\subsection{Two new lower bounds}

In this subsection we prove two new lower bounds on the outer multiset dimension. The first one solves~\cite[Problem 6.1]{klavzar-2023} as follows. 

\begin{theorem}
\label{thm:outer-lower-for-diam-2}
If a graph $G$ has diameter two, then 
$$\odim(G) \ge n(G) - \Delta(G)\,.$$
Moreover, for every $k\ge 1$, there exists a diameter two graph of order $2k+1$ for which the bound is sharp. 
\end{theorem}

\begin{proof}
Set $n = n(G)$ and $\Delta = \Delta(G)$. Suppose, to the contrary, that $S$ is an outer multiset basis of $G$ with $|S| \le n-\Delta-1$. We investigate two cases.

\medskip
\noindent
\textbf{Case 1}: $|S| \le n-\Delta-2$. \\
In this case, at least $\Delta+2$ vertices are outside of $S$. For every $v \in V(G) \setminus S$, we have $m_G(v|S) = \msl 1^a,2^{|S|-a}\msr$, where $0 \le a \le \Delta$. Hence, there are $\Delta+1$ possible multisets for $v$, and thus there must be two vertices with the same multisets, a contradiction.

\medskip
\noindent
\textbf{Case 2}: $|S| = n-\Delta-1$. \\
Now $|V(G) \setminus S| = \Delta+1$. Since there are $\Delta+1$ possible multisets, each vertex in $V(G) \setminus S$ must receive a unique multiset. Thus, there are vertices $u,v \in V(G) \setminus S$ whose entry $1$ in the multiset has multiplicity $0$ and $\Delta$, respectively. This means that $N(u) \cap S = \emptyset$ while $N(v) \subseteq S$. Hence $uv \notin E(G)$ and $N(u) \cap N(v) = \emptyset$. However, since $G$ has diameter $2$, we have $d(u,v) = 2$, and they must have a common neighbor, a contradiction.

From both cases, we conclude that $|S| \ge n - \Delta$.

To demonstrate that the bound is sharp, consider the graph construction described below. For every integer $k \ge 1$, let the graph $G_k$ be defined as follows. The vertex set of $G_k$ is $\{u_1,\dots,u_k\} \cup \{v_1,\dots,v_k\} \cup \{w\}$. As for the edge set of $G_k$, first connect $v_j$ to $u_i$ for every $1 \le j \le i \le k$. Next, add the following edges, let's refer to them as {\em cross-edges}:
$$
\{u_iu_{k-i+1} : 1 \le i \le \lfloor k/2 \rfloor-1\} \cup
\begin{cases}
    \{u_{k/2}u_{k/2+1}\}, &k \ \text{is even}, \\
    \{u_{\lfloor k/2 \rfloor}u_{\lfloor k/2 \rfloor + 2}\}, &k \ \text{is odd},
\end{cases}
$$
and
$$
\{v_jv_{k-j+2} : 2 \le j \le \lceil k/2 \rceil-1\} \cup
\begin{cases}
    \{v_{k/2}v_{k/2+2}\}, &k \ \text{is even}, \\
    \{v_{\lceil k/2 \rceil}v_{\lceil k/2 \rceil + 1}\}, &k \ \text{is odd}.
\end{cases}
$$
Finally, we connect $w$ to $v_j$ for $j \in [k]$. Note that $G_1 \cong P_3$ and that $G_2$ is the house graph. The graph $G_7$ can be seen in Fig.~\ref{fig:G7}. 

\begin{figure}[ht!]
\begin{center}
\begin{tikzpicture}[scale=0.7,style=thick]
\tikzstyle{every node}=[draw=none,fill=none]
\def\vr{4pt} 

\begin{scope}[yshift = 0cm, xshift = 0cm]
\path (0,0) coordinate (v1);
\path (2,0) coordinate (v2);
\path (4,0) coordinate (v3);
\path (6,0) coordinate (v4);
\path (8,0) coordinate (v5);
\path (10,0) coordinate (v6);
\path (12,0) coordinate (v7);
\path (0,4) coordinate (u1);
\path (2,4) coordinate (u2);
\path (4,4) coordinate (u3);
\path (6,4) coordinate (u4);
\path (8,4) coordinate (u5);
\path (10,4) coordinate (u6);
\path (12,4) coordinate (u7);
\path (6,-3) coordinate (w);

\foreach \i in {1,...,7} \draw (v1)--(u\i);
\foreach \i in {2,...,7} \draw (v2)--(u\i);
\foreach \i in {3,...,7} \draw (v3)--(u\i);
\foreach \i in {4,...,7} \draw (v4)--(u\i);
\foreach \i in {5,...,7} \draw (v5)--(u\i);
\foreach \i in {6,...,7} \draw (v6)--(u\i);
\foreach \i in {7,...,7} \draw (v7)--(u\i);
\foreach \i in {1,...,7} \draw (w)--(v\i);
\draw (u1) .. controls (1,6) and (11,6) .. (u7);
\draw (u2) .. controls (3,5.5) and (9,5.5) .. (u6);
\draw (u3) .. controls (5,5) and (7,5) .. (u5);
\draw (v2) .. controls (4,-1) and (10,-1) .. (v7);
\draw (v3) .. controls (5.5,-0.7) and (8.5,-0.7) .. (v6);
\draw (v4) .. controls (6.5,-0.5) and (7.5,-0.5) .. (v5);

\foreach \i in {1,...,7} \draw (v\i)  [fill=white] circle (\vr);
\foreach \i in {1,...,7} \draw (u\i)  [fill=white] circle (\vr);
\draw (w)  [fill=white] circle (\vr);
\draw[left] (v1)++(-0.1,0.1) node {$v_1$};
\draw[left] (v2)++(-0.1,0.1) node {$v_2$};
\draw[left] (v3)++(-0.1,0.1) node {$v_3$};
\draw[left] (v4)++(-0.1,0.1) node {$v_4$};
\draw[left] (v5)++(-0.1,0.1) node {$v_5$};
\draw[left] (v6)++(-0.1,0.1) node {$v_6$};
\draw[left] (v7)++(-0.1,0.1) node {$v_7$};
\draw[left] (u1)++(-0.1,0.1) node {$u_1$};
\draw[left] (u2)++(-0.1,0.1) node {$u_2$};
\draw[left] (u3)++(-0.1,0.1) node {$u_3$};
\draw[left] (u4)++(-0.1,0.1) node {$u_4$};
\draw[left] (u5)++(-0.1,0.1) node {$u_5$};
\draw[left] (u6)++(-0.1,0.1) node {$u_6$};
\draw[left] (u7)++(-0.1,0.1) node {$u_7$};
\draw[below] (w)++(0.0,-0.1) node {$w$};
\end{scope}

\end{tikzpicture}
\end{center}
\caption{Graph $G_7$}
\label{fig:G7}
\end{figure}
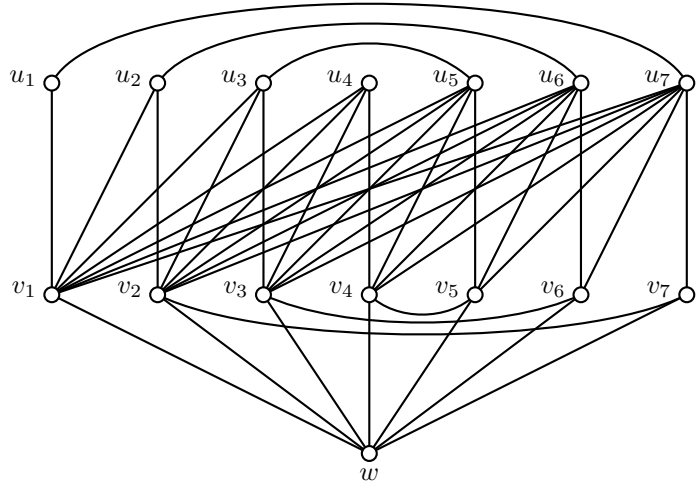

Clearly, $n(G_k) = 2k+1$ and $\Delta(G_k) = k+1$. We next claim that $\diam(G_k) = 2$. Note first that $\ecc_{G_k}(w)=2$ since $wv_i, v_1u_i\in E(G_k)$ for $i\in [k]$. Next, since $v_1u_i\in E(G_k)$ for $i\in [k]$, we have $d_{G_k}(u_i,u_j) \le 2$ for $i,j\in [k]$. Similarly, since $u_kv_i\in E(G_k)$ for $i\in [k]$, we have $d_{G_k}(v_i,v_j) \le 2$ for $i,j\in [k]$. Finally, the cross-edges ensure that $d_{G_k}(v_i,u_j) \le 2$ also holds for all $i,j\in [k]$.

Since the set $\{u_1,\dots,u_k\}$ is as an outer multiset resolving set of $G_k$, we can conclude that $G_k$ attains the bound of the theorem.
\end{proof}

In~\cite[Problem 2]{pervaiz-2025}, a problem was posed to determine, for an arbitrary graph $G$, a lower bound for $\odim(G+K_1)$ and a lower bound for $\odim(G + \overline{K_m})$ for every $m\ge 2$. We can answer these two questions with the next result.  

\begin{theorem}
If $G$ is a graph and $m \ge 1$, then
$$\odim(G+\overline{K_m}) \ge n(G) - \Delta(G) + m - 2.$$
Moreover, the bound is sharp.
\end{theorem}

\begin{proof}
    The bound trivially holds if $G$ is complete, so let us assume for the rest that $G$ is not complete. Then $\diam(G+\overline{K_m}) = 2$. Set $n = n(G)$, $\Delta = \Delta(G)$, and $V(\overline{K_m}) = \{v_1,\dots,v_m\}$. Suppose, to the contrary, that $S$ is an outer multiset resolving basis of $G+\overline{K_m}$ with $|S| \le n-\Delta+m-3$. Then at least $\Delta+3$ vertices in $G+\overline{K_m}$ are excluded from $S$. Since $v_1,\dots,v_m$ are twins, by Proposition \ref{prop:outer-twins}, we may assume, without loss of generality, that $v_1,\dots,v_{m-1}$ are elements of $S$. Note that if $m = 1$, then there are no such vertices. We consider two cases with respect to the vertex $v_m$.
    
    \medskip
    \noindent
    \textbf{Case 1:} $v_m \notin S$. \\
    There are at least $\Delta+3$ vertices in $V(G) \setminus S$. For every $x \in V(G) \setminus S$, we have $m_{G+\overline{K_m}}(x|S) = \msl 1^a, 2^{|S|-a} \msr$. Since the neighbors of $x$ in $S$ consist of $m-1$ vertices in $\overline{K_m}$ and a maximum of $\Delta$ vertices in $G$, we have $m-1 \le a \le m-1+\Delta$. Hence, there are $\Delta+1$ possible multisets for $x$, and thus there must be two vertices in $V(G) \setminus S$ with the same multisets, a contradiction.
    
    \medskip
    \noindent
    \textbf{Case 2:} $v_m \in S$. \\
    In this case, there are at least $\Delta+2$ vertices in $V(G) \setminus S$. For every $x \in V(G) \setminus S$, we again have $m_{G+\overline{K_m}}(x|S) = \msl 1^a, 2^{|S|-a} \msr$. However, since $x$ is adjacent to at least $m$ neighbors in $S$, namely $v_1,\dots,v_m$, its remaining neighbors in $S$ belong to $G$, and so $m \le a \le m+\Delta$. This again yields $\Delta+1$ possible multisets for $x$, and thus at least two vertices in $V(G) \setminus S$ share the same multisets, another contradiction.
    
    From both cases, we conclude that $|S| \ge n-\Delta+m-2$.

    The bound is sharp, as demonstrated by the generalized fan graphs $F_{s,m} = P_{s} + \overline{K_m}$. Clearly, $n(P_s) = s$ and $\Delta(P_s) = 2$. Indeed, Theorem~\ref{thm:generalized-fan} gives $\odim(F_{s,m}) = s+m-4$ for $m \ge 5$.
\end{proof}

\section{Local multiset dimension}
\label{sec:local}

In this section we discuss the local multiset dimension and divide the material into three subsections. In the first subsection we present exact results, bounds, and a related concept of the local outer multiset dimension. In the subsequent section we present known results on the local multiset dimension of different graph products. In the concluding subsection we present new results on graphs with the local multiset dimension equal to two. In particular, we characterize block graphs which have this property. 

\subsection{Exact values, bounds, local outer multiset dimension}

The complete graph has $\ldim(K_n) = \infty$ for every $n \ge 3$. The results for complete multipartite graphs, cycles, and wheels are as follows.

\begin{theorem}{\rm \cite[Theorem 2.6]{alfarisi-2019}}
If $l\in [k]$ and $n_l \ge k-1$, then $\ldim(K_{n_1,\dots,n_k})=\frac{k(k-1)}{2}$.
\end{theorem}

\begin{theorem}{\rm \cite[Theorem 10]{simanjuntak-2026+}}
\label{thm:local-cycle}
If $n \ge 3$, then 
$$
\ldim(C_n) =
\begin{cases}
    1; &n \ \text{is even}\,, \\
    3; &n \ge 7 \ \text{is odd}\,, \\
    \infty; & n\in \{3,5\}\,.
\end{cases}
$$
\end{theorem}

\begin{theorem}{\rm \cite[Theorem 15]{simanjuntak-2026+}}
If $n \ge 3$, then
$$
\ldim(W_{n+1}) =
\begin{cases}
    3; &n\in\{4,6\}\,,\\
    \lceil \frac{n}{4}\rceil; &n\ge 8 \; \text{even}\,, \\
    \infty; &\text{otherwise}\,.
\end{cases}
$$
\end{theorem}

Furthermore, results on paths, stars, complete $k$-ary trees, and caterpillars are also presented in \cite{alfarisi-2019}, all of which have a local multiset dimension $1$. Indeed, it has been proven later in \cite{simanjuntak-2026+} that all bipartite graphs attain this value. This follows from the property that if $u \in V(G)$ constitutes a local multiset basis of $G$, then for any edge $vw \in E(G)$, $d(u,v)$ and $d(u,w)$ have different parities.

\begin{theorem}{\rm \cite[Theorem 8]{simanjuntak-2026+}}
\label{th:local-equal-1}
Let $G$ be a graph. Then $\ldim(G) = 1$ if and only if $G$ is bipartite.
\end{theorem}

This result also partially generalizes previous contributions on unicyclic graphs \cite{adawiyah-2019}, including the pan graph, sunlet graph, and a cycle with two pendant edges. In general, the following formula has been provided.

\begin{theorem}{\rm \cite[Theorem 2.1]{alfarisi-2023b}}
If $G$ is a non-cycle unicyclic graph, then
$$
\ldim(G) =
\begin{cases}
    1; & $G$ \text{ bipartite}\,,\\
    2; & \text{otherwise}\,.\\
\end{cases}
$$
\end{theorem}

Continuing in this direction, the case of bicyclic graphs with disjoint cycles was also studied, while the other case remains open for investigation.

\begin{theorem}{\rm \cite[Theorem 2.2]{alfarisi-2023b}}
If $G$ is a bicyclic graph with two disjoint cycles of order $p_1$ and $p_2$, then
$$
\ldim(G) =
\begin{cases}
    1; & \text{$p_1$ and $p_2$ are even}\,,\\
    2; & \text{otherwise}\,.\\
\end{cases}
$$
\end{theorem}

Characterizing graphs with a finite or infinite local multiset dimension remains a major open problem in this subject. To this end, a simple characterization involving twins has been stated as follows. 

\begin{theorem}{\rm \cite[Theorem 4]{simanjuntak-2026+}}
If a graph $G$ has $\ldim(G) < \infty$, then every maximal clique of $G$ contains at most two simplicial vertices. Moreover, if a maximal clique of $G$ contains two simplicial vertices, then exactly one must be in every local multiset resolving set of $G$.
\end{theorem}

Several remarkable bounds for the local multiset dimension of a graph have been found, with some relation to other graph parameters such as the diameter, the clique number, and the chromatic number.

\begin{theorem}{\rm \cite[Theorem 11]{simanjuntak-2026+}}
If $G$ is a non-trivial graph, then $\ldim(G) \ge \left\lceil \log_2 \omega(G) \right\rceil.$
\end{theorem}

\begin{theorem}{\rm \cite[Theorem 11]{simanjuntak-2026+}}
\label{thm:local-lower-chromatic}
Let $G$ be a graph with diameter $d \ge 2$. If $g(d,\chi(G))$ is the smallest number $k$ such that
$\binom{k+d-1}{d-1}+\binom{k+d-2}{d-1}-d+1\ge\chi(G),$
then $\ldim(G) \ge g(d,\chi(G))$. Moreover, this bound is sharp.
\end{theorem}

The sharpness of this bound is demonstrated by all bipartite graphs with diameter 2; thus, the problem of characterizing all graphs that attain the equality naturally arises. For small diameters, Theorem \ref{thm:local-lower-chromatic} turns into noticeable closed forms as follows.

\begin{corollary}{\rm \cite[Corollary 2]{simanjuntak-2026+}}
If $G$ has diameter $2$, then $\ldim(G) \ge \frac{\chi(G)}{2}$. If $G$ has diameter $3$, then $\ldim(G) \ge \sqrt{\chi(G)+2}-1$.
\end{corollary}

As stated in \cite{alfarisi-2020}, the local multiset dimension is not monotonic with respect to the number of vertices or edges in a graph. This means that adding some vertices or edges to a given graph does not necessarily increase or decrease the local multiset dimension of the original graph.
However, the following bound gives a relation between the local multiset dimension of a graph and its maximal subgraph without a leaf.

\begin{theorem}{\rm \cite[Theorem 13]{simanjuntak-2026+}}
Let $H$ be a maximal subgraph of $G$ without a leaf. If $\ldim(H) < \infty$, then $\ldim(G) \le \ldim(H)$. Moreover, this bound is sharp.
\end{theorem}

To conclude this subsection, we note that the authors in \cite{simanjuntak-2026+} extended the local multiset dimension concept to the ``outer'' version. A set $S \subseteq V(G)$ of a graph $G$ is a \emph{local outer multiset resolving set} for $G$ if, for every adjacent pair $x,y \in V(G) \setminus S$, the multisets $m_G(x|S)$ and $m_G(y|S)$ are different. The cardinality of a smallest local outer multiset resolving set is the \textit{local outer multiset dimension} of $G$, denoted by $\lodim(G)$. In their work, they present parallel results on both the local multiset dimension and the local outer multiset dimension, among which we state the following.

\begin{theorem}{\rm \cite[Theorem 10]{simanjuntak-2026+}}
If $n \ge 3$, then 
$$
\lodim(C_n) =
\begin{cases}
    1; &n \ \text{is even}\,, \\
    2; &n \ \text{is odd}\,.
\end{cases}
$$
\end{theorem}

\begin{theorem}{\rm \cite[Theorem 15]{simanjuntak-2026+}}
If $n \ge 3$, then
$$
\lodim(W_{n+1}) =
\begin{cases}
    3; & n\in \{3,4,6\},\\
    \lceil\frac{n}{4}\rceil; &n\ge 8 \; \text{even, or } n \equiv 1 \bmod 4, \\
    \lceil\frac{n}{4}\rceil+1; &\text{otherwise}.
\end{cases}
$$
\end{theorem}

\subsection{Local multiset dimension of graph products}

The authors in~\cite{adawiyah-2019, alfarisi-2019} posed the problem of finding the local multiset dimension of graph products. To date, several results have been presented in this direction. We begin with the Cartesian product of two graphs.

\begin{theorem}{\rm \cite[Lemma 3.2]{alfarisi-2020}}
If $G_1$ and $G_2$ are connected graphs, then $$\ldim(G_1 \cp G_2) \ge \min\{\ldim(G_1),\ldim(G_2)\}.$$
\end{theorem}

\begin{theorem}
If $G$ is a connected graph, then the following statements hold. 
\begin{enumerate}
    \item {\rm \cite[Theorem 3.1]{alfarisi-2020}} For every $n \ge 1$, $\ldim(G \cp P_n) = \ldim(G)$.
    \item {\rm \cite[Theorem 3.2]{alfarisi-2020}} For every tree $T$, $\ldim(G \cp T) \le \ldim(G)$.
\end{enumerate}
\end{theorem}

We next state some results for the rooted product of graphs. Let $G$ be a graph with $V(G) = \{v_1,\dots,v_n\}$, and let $\mathcal{H}$ be a set of $n$ disjoint graphs $H_1, \dots, H_n$, where a vertex in each $H_i$ is chosen as its root. The \textit{rooted product} of $G$ by $\mathcal{H}$, denoted by $G(\mathcal{H})$, is the graph obtained by identifying the root of $H_i$ and $v_i$ for every $i$. In the case where all $H_i$ are isomorphic to the same graph $H$ with root $v$, we write $G \circ_v H$ instead. Contributions involving trees and triangles are discussed in \cite{alfarisi-2025}, and results involving paths and cycles are listed below.

\begin{theorem}{\rm \cite[Theorem 1]{alfarisi-2023a}}
If $n \ge 2$, $m \ge 3$, and $v \in V(C_m)$, then
$$
\ldim(P_n \circ_v C_m) =
\begin{cases}
    1; &\text{$m$ is even}\,, \\
    n; &\text{$m>3$ is odd, or $m=3$ and $n$ is odd}\,, \\
    n+1; &\text{$m=3$ and $n$ is even}\,.
\end{cases}
$$
\end{theorem}

\begin{theorem}{\rm \cite[Theorem 4]{alfarisi-2023a}}
If $n \ge 3$, $m \ge 3$, and $v \in V(C_m)$, then
$$
\ldim(C_n \circ_v C_m) =
\begin{cases}
    1; &\text{both $n,m$ are even}\,, \\
    2; &\text{exactly one of $n,m$ is even}\,, \\
    n; &\text{both $n,m$ are odd}\,.
\end{cases}
$$
\end{theorem}

The \textit{corona product} $G \odot \mathcal{H}$ of a graph $G$ and a finite collection $\mathcal{H}=\{H_1,\dots,H_n\}$ of graphs is the graph obtained by joining the $i$-th vertex of $G$ to every vertex of $H_i$. As another special case of the rooted product, one can see that $G \odot \mathcal{H} = G(\{H_1+K_1,\dots,H_n+K_1\})$ where, for each $H_i+K_1$, its universal vertex is chosen as the root. In \cite{simanjuntak-2026+}, the authors discussed the corona products involving complete graphs whose local multiset dimensions are finite. In addition, \cite{alfarisi-2024b} presents a study on the local multiset dimension on corona products involving trees. 

\begin{theorem}{\rm \cite[Theorem 7]{simanjuntak-2026+}}
Let $G$ be a graph with $n(G) \ge 1$ and $m_i \ge 1$ for every $i \in [n]$. Set $\mathcal{H} = \{K_{m_1},\dots,K_{m_n}\}$. If $\ldim(G \odot \mathcal{H}) < \infty$, then $m_i \le 2$ for every $i \in [n]$. Moreover, if $m_i \le 2$ for every $i \in [n]$, then $\ldim(G \odot \mathcal{H}) \ge m$.
\end{theorem}

The last graph operation to be considered is the following. Let $G$ be a graph with a root $v$. The \textit{amalgamation} $\Amal(G, v, m)$ is the graph obtained by taking $m$ copies of $G$ and identifying all of their roots. Several contributions to this product are found in \cite{alfarisi-2024a}. Among all, the following general bound is stated.

\begin{theorem}{\rm \cite[Lemma 0.2]{alfarisi-2024a}}
\label{thm:local-amalgamation}
If $n \ge 3$, $m \ge 2$, and $v$ is a vertex of a connected graph $G$, then
$$\ldim(\Amal(G,v,m)) \le m \cdot \ldim(G).$$
\end{theorem}

The authors in \cite{alfarisi-2024a} then continued to provide some results for the amalgamation of wheels $W_n$ and fans $F_n$ for various root vertices. In particular, they showed that if both graphs take the universal vertex as their root, then the bound in Theorem \ref{thm:local-amalgamation} becomes an equality. In more general settings, the amalgamation of a finite collection of graphs has been defined, either taking a vertex or an edge as their roots. These types of amalgamation have been discussed in \cite{simanjuntak-2026+} with a particular interest in complete graphs. Finally, some other contributions on the local multiset dimension seem to have been presented in \cite{shankar-2025}.



\subsection{Graphs with local multiset dimension two}

By Theorem~\ref{th:local-equal-1}, $\ldim(G)=1$ if and only if $G$ is a bipartite graph. Hence, it is natural to consider characterizing those graphs $G$ for which $\ldim(G)=2$. We present here some partial results in this direction, and to this end, we first show some basic observations. Clearly, a first situation that immediately follows is that if $\ldim(G)=2$, then $G$ is not bipartite.

\begin{remark}
\label{rem:basis-2-even-dist}
If $G$ is a non-bipartite graph with a local multiset basis $S=\{u,v\}$, then $d_G(u,v)$ is even.
\end{remark}

\begin{proof}
If $d_G(u,v)$ is odd, then there are two adjacent vertices $x,y$ in a shortest $u,v$-path (it is possible that $uv\in E(G)$) such that $m_G(x|S)=\msl \lfloor d_G(u,v) / 2\rfloor, \lceil d_G(u,v) / 2\rceil\msr=m_G(y|S)$, which is not possible.
\end{proof}

The remark above leads to the fact that if $S=\{u,v\}$ is a local multiset basis, then $u$ and $v$ cannot de adjacent. Some sporadic examples of graphs with local multiset dimension two are drawn in Figure~\ref{fig:examples}.

\begin{figure}[ht]
    \centering

\begin{tikzpicture}[
    vertex/.style={circle, draw=black, thick, minimum size=6pt, inner sep=0pt},
    filled/.style={vertex, fill=black},
    empty/.style={vertex, fill=white}
]

\begin{scope}[shift={(0,-0.5)}, scale=0.8]
    \node[filled] (a1) at (-0.25, 0) {};
    \node[empty]  (a2) at (1.75, 0) {};
    \node[empty]  (a3) at (0.75, -1.5) {};
    \draw (a1) -- (a2) -- (a3) -- (a1);
    
    \node[empty] (a_mid) at (0.75, -3) {};
    \draw (a3) -- (a_mid);
    
    \node[empty] (a4) at (0.75, -4.5) {};
    \node[empty] (a5) at (-0.25, -6) {};
    \node[filled] (a6) at (1.75, -6) {};
    \draw (a_mid) -- (a4);
    \draw (a4) -- (a5) -- (a6) -- (a4);
\end{scope}

\begin{scope}[shift={(4.5,0)}]
    \node[filled] (T) at (0.75, 0) {};
    \node[filled] (B) at (0.75, -6) {};
    
    \node[empty] (L1) at (-0.5,-1) {}; \node[empty] (L2) at (-1.5,-2) {}; \node[empty] (L3) at (-0.5,-3) {};
    \node[empty] (L4) at (-1.5,-4) {}; \node[empty] (L5) at (-0.5,-5) {};
    \draw (T) -- (L1) -- (L2) -- (L3) -- (L4) -- (L5) -- (B);
    \draw (L1) -- (L3) -- (L5);

    \node[empty] (ML1) at (0.1,-1.2) {}; \node[empty] (ML2) at (0.5,-2.2) {}; \node[empty] (ML3) at (0.1,-3) {};
    \node[empty] (ML4) at (0.5,-3.8) {}; \node[empty] (ML5) at (0.1,-4.8) {};
    \draw (T) -- (ML1) -- (ML2) -- (ML3) -- (ML4) -- (ML5) -- (B);

    \node[empty] (MR1) at (1.4,-1.2) {}; \node[empty] (MR2) at (1.0,-2.2) {}; \node[empty] (MR3) at (1.4,-3) {};
    \node[empty] (MR4) at (1.0,-3.8) {}; \node[empty] (MR5) at (1.4,-4.8) {};
    \draw (T) -- (MR1) -- (MR2) -- (MR3) -- (MR4) -- (MR5) -- (B);

    \node[empty] (R1) at (2.0,-1) {}; \node[empty] (R2) at (3.0,-2) {}; \node[empty] (R3) at (2.0,-3) {};
    \node[empty] (R4) at (3.0,-4) {}; \node[empty] (R5) at (2.0,-5) {};
    \draw (T) -- (R1) -- (R2) -- (R3) -- (R4) -- (R5) -- (B);
    \draw (R1) -- (R3) -- (R5);
\end{scope}

\begin{scope}[shift={(10.5,-0.5)}, scale=0.8]
    \node[empty]  (c1) at (0,0) {};
    \node[empty]  (c2) at (1.2,-1.5) {};
    \node[empty]  (c3) at (1.2,-4.5) {};
    \node[filled] (c4) at (0,-6) {}; 
    \node[empty]  (c5) at (-1.2,-4.5) {};
    \node[filled] (c6) at (-1.2,-1.5) {}; 
    
    \draw (c1) -- (c2) -- (c3) -- (c4) -- (c5) -- (c6) -- (c1);
    \draw (c3) -- (c5);
\end{scope}

\end{tikzpicture}
    \caption{Three graphs $G$ with $\ldim(G)=2$. Local multiset basis are drawn in bold.}
    \label{fig:examples}
\end{figure}
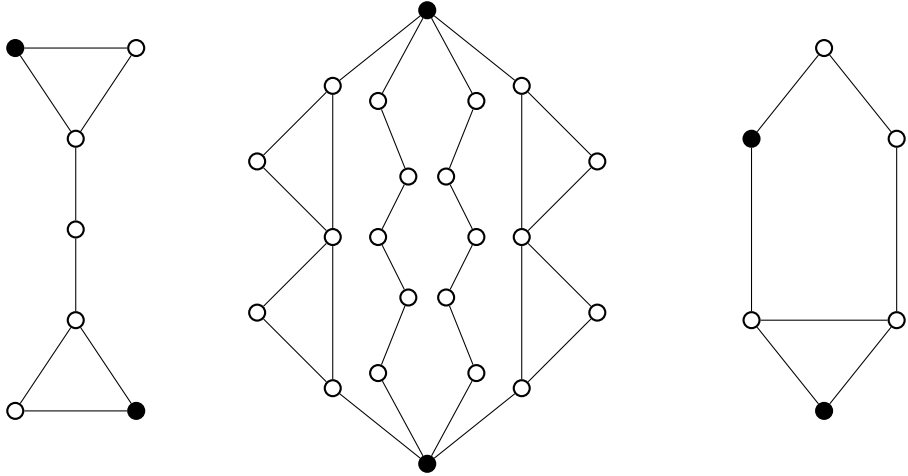

Observing the structure of such examples, one can notice that the quantity and variety of graphs $G$ satisfying $\ldim(G)=2$ is huge, and describing their structure is probably a very challenging task. In this sense, we center our attention here on the class of (not bipartite) block graphs. Recall that block graphs are the graphs in which every maximal $2$-connected subgraph is complete. From now on, we assume $q$ represents the number of maximal $2$-connected subgraphs of a block graph $G$. We may also say that $V(G)=\bigcup_{i\in [q]} V_i$, where each $V_i$ is the set of vertices of a maximal $2$-connected subgraph of $G$. Notice that $|V_i\cap V_j|\le 1$ for any distinct $i,j\in [q]$. We shall also say that the sets $V_i$, $i\in [q]$, are the blocks of $G$.

In a block graph $G$, each vertex is either a cut vertex or a simplicial vertex. A block $V_i$ is called a {\em terminal block} if it has at most one cut vertex. Trees (every block is isomorphic to $K_2)$ and complete graphs (only one block) are the most basic examples of block graphs.

In order to characterize the block graphs $G$ with $\ldim(G)=2$, we define the class of block graphs $\mathcal{B}$. A graph $G$ belongs to $\mathcal{B}$ if $G$ is a block graph with $\omega(G) = 3$, and contains a shortest path of even length which passes through exactly two vertices of each of the triangles of $G$. Examples of graphs in $\mathcal{B}$ are, for instance, all the unicyclic graphs of order at least $4$ in which the unique cycle has order three. Moreover, note that the left-hand side graph of Figure \ref{fig:examples} is a graph from the family $\mathcal{B}$.

\begin{theorem}
Let $G$ be a block graph. Then $\ldim(G)=2$ if and only if $G\in \mathcal{B}$.
\end{theorem}

\begin{proof}
Let $G$ be a block graph. We recall that each pair of vertices in $G$ is connected by a unique shortest path. 

$(\Rightarrow)$ Assume $\ldim(G)=2$ and let $S=\{x,y\}$ be a local multiset basis. Also, let $P$ be the unique shortest $x,y$-path in $G$. Notice that, by Remark \ref{rem:basis-2-even-dist}, the distance between $x$ and $y$ must be even, which means $x,y$ are not adjacent.  If $\omega(G)=2$, then $G$ is a tree, which is a bipartite graph, and then $\ldim(G)=1$, which is not possible. Hence, $\omega(G)\ge 3$. Suppose $\omega(G)>3$ and let $V_j$ be a block with at least $4$ vertices in $G$. Notice that $|V(P)\cap V_j|\le 2$ due to the structure of shortest paths in a block graph. 

Assume first $1\le |V(P)\cap V_j|\le 2$. Since $|V_j|\ge 4$, due to the uniqueness of shortest paths between vertices of a block graph, we deduce the existence of two vertices $u,v\in V_j$ (which are adjacent) such that they have the same distances to $x$ and to $y$, which contradicts the fact that $S$ is a local multiset basis. Hence, $|V(P)\cap V_j|=0$. Now, let $z$ be a vertex of $V_j$ that minimizes the distance between the vertices in $V_j$ and the vertices from the path $P$ (notice that such $z$ is unique due to the structure of block graphs). It can be hence readily observed that any two vertices of $V_j$ (that are adjacent) other than $z$ have the same distances to $x$ and to $y$, which is again not possible. Thus, we conclude that $\omega(G) = 3$ which means that all blocks of $G$ have order either $2$ or $3$.

Notice next that, if $G$ has a terminal block of order three (i.e., two of its vertices are simplicial vertices), then exactly one of its simplicial vertices must be in $S$. Thus, $G$ can have at most two terminal blocks of order three. If (for instance) $x$ belongs to a terminal block $V_i$, then clearly the cut vertex $w$ of $V_i$ belongs to the shortest $x,y$-path $P$. Also, notice that $x,w$ are hence adjacent. Assume now, there is a block (not terminal) of order $3$, say $V_r$. Suppose that $|V(P)\cap V_r|\le 1$. Let $u$ be a vertex of $V_r$ that minimizes the distance between the vertices in $V_r$ and the vertices in $P$ (notice that this vertex can be indeed a vertex of $P$, and that it is unique). Due to the structure with the unique shortest $x,y$-path $P$, the other two vertices $u',u''\in V_r$ are adjacent and have the same distance to both vertices $x$ and $y$, which is a contradiction. Thus, $|V(P)\cap V_r| = 2$ (it cannot be three due to the structure of a block graph).

As a consequence of the arguments above, it holds that $G$ satisfies $\omega(G) = 3$ and contains a shortest path which passes through exactly two vertices of each of the triangles of $G$. Thus $G\in \mathcal{B}$. 

\medskip
$(\Leftarrow)$ Assume now $G\in \mathcal{B}$. Let $P''$ be the shortest path of even length between two vertices $w,z$ in $G$ that passes through exactly two vertices from each of the triangles of $G$. We claim that the set $D=\{w,z\}$ is a local multiset resolving set for $G$. To see this, we only need to observe that any two adjacent vertices $a,b$ of $G$ are either in the shortest path $P''$, or (w.l.o.g.), in another shortest $a,b'$-path $Q$ such that $b'\in V(P'')$ and $b\in V(Q)$. This fact makes that $a$ and $b$ have distinct multiset representations with respect to $D$. Thus, we deduce that $D$ is indeed a local multiset basis, and so $\ldim(G)=2$, since it cannot be $\ldim(G)=1$ by Theorem \ref{th:local-equal-1}. This completes this implication and the whole proof.   
\end{proof} 

\section{Other two variants}

As usual in graph theoretical studies, there are always several variations of a main concept that would appear in the investigation. Until now, we have surveyed the most remarkable cases related to the multiset dimension. In this section, we consider two other variations that have been less considered. For each one of them, we are only aware of exactly one work that studies them. In addition to these two variations, the reader can also find information about an independent version of the multiset dimension. That is, in~\cite{Kumar-2023} a study has been presented on the multiset resolving sets of graphs that are also independent sets (sets of vertices that induce edgeless graphs). 

\subsection{Edge multiset dimension}
\label{sec:edge-multi}

The concept of edge metric dimension of graphs was introduced in \cite{Kelenc-2018} as a variation focused into uniquely identifying the edges of a graph, by means of a distance vector to a set of vertices. After such seminal work, this variant has become one of the most popular metric dimension parameters, which is probably based on its non-comparability with the classical metric dimension parameter. To see more about this, we simply suggest the survey \cite[Section 8]{kuziak-2021}. In this sense, it is not surprising that an edge version of the multiset dimension has also already been addressed.

If $S =\{u_1,\dots,u_k\}$ is an ordered set of vertices of a graph $G$, then the \emph{edge metric representation} of an edge $e=uv\in E(G)$ with respect to $S$ is the vector 
$$r(e|S) = (d_G(e,u_1),\dots,d_G(e,u_k))\,,$$
where the distance between a vertex $x$ and an edge $e=uv$ is $d_G(x,e)=\min\{d_G(x,u),d_G(x,v)\}$. The set $S$ is an {\em edge resolving set} for $G$ if all the vectors $r(e|S)$, $e\in E(G)$, are pairwise different. The \emph{edge metric dimension} $\edim(G)$ of $G$ is the cardinality of a smallest edge resolving set for $G$. These concepts were first presented in~\cite{Kelenc-2018}.

Now, in \cite{Ikhlaq-2023}, the authors considered the following concepts. The \emph{edge multiset representation} of an edge $e=uv\in E(G)$ with respect to $S$ is the multiset  
$$m_G(e|S) = \msl d_G(e,u_1),\dots,d_G(e,u_k)\msr\,.$$
The set $S$ is an {\em edge multiset resolving set} for $G$ if all the multisets $m_G(e|S)$, $e\in E(G)$, are pairwise different. The \emph{edge multiset dimension} $\medim(G)$ of $G$ is the cardinality of a smallest edge multiset resolving set for $G$.

The article \cite{Ikhlaq-2023} presents some preliminary results, special conditions, and bounds on the edge-multiset dimension of certain graphs. Among them, we remark the following contributions, some of which are rather similar and their proofs being equivalent to the standard multiset dimension parameter.

\begin{theorem}{\em \cite[Theorem 9]{Ikhlaq-2023}}
Let $G$ be a graph. Then $\medim(G) = 1$  if and only if $G\cong P_n$.
\end{theorem}

\begin{lemma}{\em \cite[Lemma 2]{Ikhlaq-2023}}
If $G$ is a graph, then $\medim(G) \ne 2$.
\end{lemma}

As in the case of the multiset dimension, there are graphs having no edge multiset resolving sets, and in such cases, the value of their edge multiset dimension is established as infinite. Parallel results concerning this were also given in \cite{Ikhlaq-2023}.

\begin{lemma}{\em \cite[Lemma 3]{Ikhlaq-2023}}
\label{lem:diam-two-edges}
Let $G$ be a non-trivial graph. If the distance between any vertex and any edge of $G$ is at most $2$, then $G$ does not contain an edge multiset resolving set.
\end{lemma}

Notice that this result is parallel to the case of graphs of diameter two for the multiset dimension. If a graph $G$ has diameter two, then it satisfies the condition of Lemma \ref{lem:diam-two-edges}. However, the contrary is not true. For instance, $C_6$ satisfies the conditions of this lemma, but has diameter three. Examples of basic graphs including complete graphs, wheels, fans, complete bipartite graphs were mentioned in \cite{Ikhlaq-2023} to have infinite edge multiset dimension. 

Other contributions in this topic were focused into computing the value of $\medim(G)$ for some specific graphs $G$. Some of these results are next mentioned.

\begin{theorem}{\em \cite[Theorem 12]{Ikhlaq-2023}}
If $n\ge 7$, then $\medim(C_n) = 3$.
\end{theorem}

\begin{theorem}{\em \cite[Theorem 21]{Ikhlaq-2023}}
The edge multiset dimension of a complete $r$-ary tree is finite if and only if $r = 1$ or $r=2$. Moreover, if $T$ is a complete binary tree of height $h$, then $\medim(T) = 2^h - 1$.
\end{theorem}

Other families of graphs considered in \cite{Ikhlaq-2023} include the dragon graphs, the comb product of graphs, some caterpillar graphs, kayak paddle graphs, lobster graphs, and some families of trees.

\medskip
Finally, the work \cite{Ikhlaq-2023} focuses on making a comparison between $\medim(G)$ and $\mdim(G)$. Such a comparison is based on the relationship that exists between the classical metric dimension and the edge metric dimension. That is, it is already known that any of the possibilities $\dim(G)<\edim(G)$, $\dim(G)=\edim(G)$, or $\dim(G)>\edim(G)$ is possible for infinite classes of graphs (see \cite{Knor-2021,Knor-2022}). In \cite{Ikhlaq-2023}, the authors claim that an analogous situation occurs for the multiset dimension and the edge multiset dimension, and present examples of graphs $G$ satisfying each of the three situations $\mdim(G)<\medim(G)$, $\mdim(G)=\medim(G)$, or $\mdim(G)>\medim(G)$. However, such examples that appear in \cite{Ikhlaq-2023} are not correct. To see this, we have checked by using a brute force computer search that both graphs $G$ appearing in \cite[Figures 17 and 18]{Ikhlaq-2023} satisfy that $\mdim(G)=\medim(G)$, which contradicts what they have claimed. Specifically, in the graph from \cite[Figure 17]{Ikhlaq-2023}, the vertices $u_1,u_2,u_6$ form a multiset basis, while the vertices $u_1,u_3,u_6$ form an edge multiset basis. Also, in \cite[Figure 18]{Ikhlaq-2023}, the vertices $v_1,v_3,v_{13}$ form a multiset basis, while the vertices $v_1,v_2,v_{13}$ form an edge multiset basis.

Based on this fact, one might think that all three situations $\mdim(G)<\medim(G)$, $\mdim(G)=\medim(G)$, or $\mdim(G)>\medim(G)$ could not be possible, but indeed they are. The equality $\mdim(G) = \medim(G)$ is satisfied by the cycles with at least 7 vertices, in which both parameters admit the value $3$ \cite{saenpholphat-2009, Ikhlaq-2023}. The case $\mdim(G) < \medim(G)$ is fulfilled by the graph $G$ depicted in Figure \ref{fig:example mdim < medim}. Here, the multiset basis of $G$ is formed by solid black vertices (there are 3 of them), while the edge multiset basis is formed by the circled vertices (there are 4). Some other examples regarding this situation appear in the class of graphs obtained from the subdivision of complete graphs. To see this, given the integers $n\ge 3$ and $k\ge 1$, the \textit{subdivision graph} $K_n^k$ is obtained from a complete graph $K_n$ by subdividing each edge $k$ times. By making some computer search, we have obtained the values appearing in Tables \ref{tab:subdiv} and \ref{tab:subdiv-edge}, for the multiset and edge multiset dimensions of these graphs. There we can see that for instance, if $G$ is $K_5^2$ or $K_5^6$ (among other ones), then $\mdim(G) < \medim(G)$. 

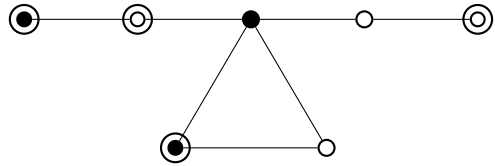
\begin{figure}[h]
\centering
\begin{tikzpicture}[
    scale=1,
    vertex/.style={circle, draw=black, thick, minimum size=6pt, inner sep=0pt},
    doublevertex/.style={circle, draw=black, thick, minimum size=8pt, inner sep=0pt, outer sep=-1pt},
    empty/.style={vertex, fill=white},
    setA/.style={vertex, fill=black},            
    setB/.style={doublevertex, fill=white, double, double distance=2pt}, 
    both/.style={doublevertex, fill=black, double, double distance=2pt}
]

    \node[both] (1) at (0,0) {};          
    \node[setB] (2) at (1.5,0) {};          
    \node[setA] (3) at (3,0) {};          
    \node[empty] (4) at (4.5,0) {};         
    \node[setB] (5) at (6,0) {};          
    \node[both] (6) at (2,-1.7) {};       
    \node[empty] (7) at (4,-1.7) {};      

    \draw (1)--(2)--(3)--(4)--(5);
    \draw (3)--(6)--(7)--(3);
\end{tikzpicture}
\caption{A graph $G$ with $\mdim(G) < \medim(G) < \infty$.}
\label{fig:example mdim < medim}
\end{figure}

\begin{table}[ht]
    \centering
    \begin{tabular}{|c|c|c|c|c|c|c|c|c|c|c|c|c|c|c|c|}
    \hline
        $n/k$ & $1$ & $2$ & $3$ & $4$ & $5$ & $6$ & $7$ & $8$ & $9$ & $10$ & $11$ & $12$ & $13$ & $14$ & $15$ \\ \hline
        $3$ & $3$ & $3$ & $3$ & $3$ & $3$ & $3$ & $3$ & $3$ & $3$ & $3$ & $3$ & $3$ & $3$ & $3$ & $3$\\ \hline
        $4$ & $3$ & $4$ & $3$ & $3$ & $3$ & $3$ & $3$ & $3$ & $3$ & $3$ & $3$ & $3$ & $3$ & $3$ & $3$  \\ \hline
        $5$ & $\infty$ & $4$ & $4$ & $4$ & $4$ & $3$ & $4$ & $3$ & $4$ & $3$ & $4$ & $3$ & $3$ & $3$ & $3$  \\
    \hline
    \end{tabular}
    \caption{Multiset dimension of some instances of $K_n^k$.}
    \label{tab:subdiv}
\end{table}

\begin{table}[ht]
    \centering
    \begin{tabular}{|c|c|c|c|c|c|c|c|c|c|c|c|c|c|c|c|}
    \hline
        $n/k$ & $1$ & $2$ & $3$ & $4$ & $5$ & $6$ & $7$ & $8$ & $9$ & $10$ & $11$ & $12$ & $13$ & $14$ & $15$ \\ \hline
        $3$ & $\infty$ & $3$ & $3$ & $3$ & $3$ & $3$ & $3$ & $3$ & $3$ & $3$ & $3$ & $3$ & $3$ & $3$ & $3$\\ \hline
        $4$ & $\infty$ & $4$ & $3$ & $3$ & $3$ & $3$ & $3$ & $3$ & $3$ & $3$ & $3$ & $3$ & $3$ & $3$ & $3$  \\ \hline
        $5$ & $\infty$ & $6$ & $4$ & $4$ & $4$ & $4$ & $4$ & $4$ & $4$ & $4$ & $4$ & $4$ & $4$ & $4$ & $3$  \\
    \hline
    \end{tabular}
    \caption{Edge multiset dimension of some instances of $K_n^k$.}
    \label{tab:subdiv-edge}
\end{table}

For the last case $\mdim(G) > \medim(G)$, we developed exhaustive search computations and found all connected graphs $G$ on 8 or 9 vertices satisfying this property, given that both parameters are finite. These graphs are listed in Figure \ref{fig:examples mdim > medim}, and they are indeed the smallest possible graphs satisfying the inequality $\mdim(G) > \medim(G)$. For each of these graphs, the multiset basis is formed by solid black vertices, while the edge multiset basis is formed by the circled vertices. To achieve this, we utilized the comprehensive combinatorial graph catalogs maintained by Brendan D. McKay. Specifically, we used the dataset of all non-isomorphic connected graphs with up to 9 vertices. These graph datasets are generated using the nauty and Traces software packages \cite{mckay-2014}, which are standard tools for canonical labeling and generating exhaustive lists of non-isomorphic graphs. The complete dataset is also available in \url{https://users.cecs.anu.edu.au/~bdm/data/graphs.html}.

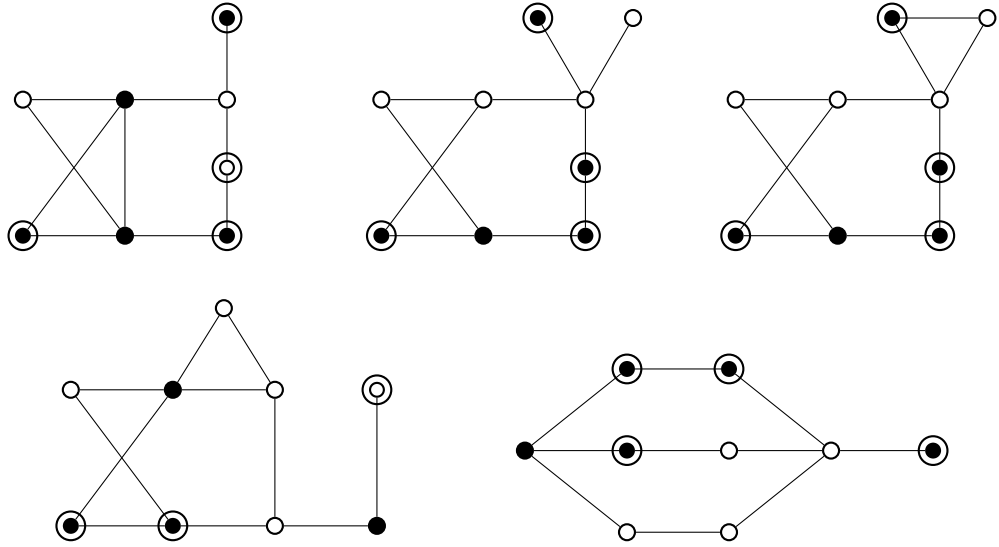
\begin{figure}[h]
\centering
\begin{tikzpicture}[
    scale=0.9,
    rotate=-90,
    xscale=-1,
    vertex/.style={circle, draw=black, thick, minimum size=6pt, inner sep=0pt},
    doublevertex/.style={circle, draw=black, thick, minimum size=8pt, inner sep=0pt, outer sep=-1pt},
    empty/.style={vertex, fill=white},
    setA/.style={vertex, fill=black},            
    setB/.style={doublevertex, fill=white, double, double distance=2pt}, 
    both/.style={doublevertex, fill=black, double, double distance=2pt}
]

    \node[both] (1) at (0,0) {};
    \node[empty] (2) at (2,0) {};
    \node[setA] (3) at (0,1.5) {};
    \node[setA] (4) at (2,1.5) {};
    \node[both] (5) at (0,3) {};
    \node[setB] (6) at (1,3) {};
    \node[empty] (7) at (2,3) {};
    \node[both] (8) at (3.2,3) {};

    \draw (1)--(3)--(5)--(6)--(7)--(8);
    \draw (1)--(4)--(7);
    \draw (3)--(2)--(4)--(3);
\end{tikzpicture}
\hspace{1.4cm}
\begin{tikzpicture}[
    scale=0.9,
    rotate=-90,
    xscale=-1,
    vertex/.style={circle, draw=black, thick, minimum size=6pt, inner sep=0pt},
    doublevertex/.style={circle, draw=black, thick, minimum size=8pt, inner sep=0pt, outer sep=-1pt},
    empty/.style={vertex, fill=white},
    setA/.style={vertex, fill=black},            
    setB/.style={doublevertex, fill=white, double, double distance=2pt}, 
    both/.style={doublevertex, fill=black, double, double distance=2pt}
]

    \node[both] (1) at (0,0) {};
    \node[empty] (2) at (2,0) {};
    \node[setA] (3) at (0,1.5) {};
    \node[empty] (4) at (2,1.5) {};
    \node[both] (5) at (0,3) {};
    \node[both] (6) at (1,3) {};
    \node[empty] (7) at (2,3) {};
    \node[both] (8) at (3.2,2.3) {};
    \node[empty] (9) at (3.2,3.7) {};

    \draw (1)--(3)--(5)--(6)--(7)--(8);
    \draw (1)--(4)--(7)--(9);
    \draw (3)--(2)--(4);
\end{tikzpicture}
\hspace{0.8cm}
\begin{tikzpicture}[
    scale=0.9,
    rotate=-90,
    xscale=-1,
    vertex/.style={circle, draw=black, thick, minimum size=6pt, inner sep=0pt},
    doublevertex/.style={circle, draw=black, thick, minimum size=8pt, inner sep=0pt, outer sep=-1pt},
    empty/.style={vertex, fill=white},
    setA/.style={vertex, fill=black},            
    setB/.style={doublevertex, fill=white, double, double distance=2pt}, 
    both/.style={doublevertex, fill=black, double, double distance=2pt}
]

    \node[both] (1) at (0,0) {};
    \node[empty] (2) at (2,0) {};
    \node[setA] (3) at (0,1.5) {};
    \node[empty] (4) at (2,1.5) {};
    \node[both] (5) at (0,3) {};
    \node[both] (6) at (1,3) {};
    \node[empty] (7) at (2,3) {};
    \node[both] (8) at (3.2,2.3) {};
    \node[empty] (9) at (3.2,3.7) {};

    \draw (1)--(3)--(5)--(6)--(7)--(8)--(9);
    \draw (1)--(4)--(7)--(9);
    \draw (3)--(2)--(4);
\end{tikzpicture}

\vspace{0.6cm}

\begin{tikzpicture}[
    scale=0.9,
    rotate=-90,
    xscale=-1,
    vertex/.style={circle, draw=black, thick, minimum size=6pt, inner sep=0pt},
    doublevertex/.style={circle, draw=black, thick, minimum size=8pt, inner sep=0pt, outer sep=-1pt},
    empty/.style={vertex, fill=white},
    setA/.style={vertex, fill=black},            
    setB/.style={doublevertex, fill=white, double, double distance=2pt}, 
    both/.style={doublevertex, fill=black, double, double distance=2pt}
]

    \node[both] (1) at (0,0) {};
    \node[empty] (2) at (2,0) {};
    \node[both] (3) at (0,1.5) {};
    \node[setA] (4) at (2,1.5) {};
    \node[empty] (5) at (0,3) {};
    \node[empty] (6) at (2,3) {};
    \node[empty] (7) at (3.2,2.25) {};
    \node[setA] (8) at (0,4.5) {};
    \node[setB] (9) at (2,4.5) {};

    \draw (1)--(3)--(5)--(8)--(9);
    \draw (3)--(2)--(4)--(6)--(5);
    \draw (1)--(4)--(7)--(6);
\end{tikzpicture}
\hspace{1.4cm}
\begin{tikzpicture}[
    scale=0.9,
    rotate=-90,
    vertex/.style={circle, draw=black, thick, minimum size=6pt, inner sep=0pt},
    doublevertex/.style={circle, draw=black, thick, minimum size=8pt, inner sep=0pt, outer sep=-1pt},
    empty/.style={vertex, fill=white},
    setA/.style={vertex, fill=black},            
    setB/.style={doublevertex, fill=white, double, double distance=2pt}, 
    both/.style={doublevertex, fill=black, double, double distance=2pt}
]

    \node[setA] (1) at (0,0) {};
    \node[both] (2) at (-1.2,1.5) {};
    \node[both] (3) at (0,1.5) {};
    \node[empty] (4) at (1.2,1.5) {};
    \node[both] (5) at (-1.2,3) {};
    \node[empty] (6) at (0,3) {};
    \node[empty] (7) at (1.2,3) {};
    \node[empty] (8) at (0,4.5) {};
    \node[both] (9) at (0,6) {};

    \draw (1)--(2)--(5)--(8)--(9);
    \draw (1)--(3)--(6)--(8);
    \draw (1)--(4)--(7)--(8);
\end{tikzpicture}
\caption{All connected graphs $G$ on 8 or 9 vertices with $\medim(G) < \mdim(G) < \infty$.}
\label{fig:examples mdim > medim}
\end{figure}

Notice that, among the examples shown in Figure \ref{fig:examples mdim > medim}, there are graphs which are bipartite, as well as, some which are not bipartite. Hence, in contrast to the case of the relationship between the classical metric dimension and the edge metric dimension of bipartite graphs,  the bipartition properties of a graph $G$ seem to be not related to the fact that $\medim(G) < \mdim(G)$. Recall that, as proved in \cite[Lemma 2.1]{Kelenc-2023}, if $G$ is a connected bipartite graph, then any resolving set of $G$ is also an edge resolving set for $G$, and so, $\edim(G) < \dim(G)$.

\subsection{$k$-multiset antidimension}
\label{sec:antidim}

In this subsection, we discuss a multiset related structure called $k$-multiset antiresolving set. It can be seen as a dual concept regarding the multiset resolving sets. That is, while in the multiset resolving sets we are looking to uniquely identify the vertices of the graph through a multiset of distances, in the $k$-multiset antiresolving sets we are interested in precisely not identifying the vertices of the graph. The formal related concepts are as follows. 

Consider a vertex $x$ and a set of vertices $S$ of a connected graph $G$. The \emph{anonymity set} of a user $x$ with respect to $S$ is given by $S_{[x]} = \{v \in V(G) \,:\, m_G(x|S) = m_G(v|S)\}$. The set $S$ is called a \emph{$k$-multiset antiresolving set} ($k$-MARS for short) if $k$ is the largest integer such that for every vertex $x \in V(G) \setminus S$ it holds that $|S_{[x]}| \geq k$. The $k$-\emph{multiset antidimension}, denoted $\adimms_k(G)$, is the cardinality of a smallest $k$-MARS of $G$. In Figure \ref{fig:Q_3} we show a graph in order to exemplify the concepts above. Consider the $3$-cube $Q_3$ and observe that any two vertices at distance at most $2$ in $Q_3$ form a $2$-MARS, while two vertices at distance $3$ in $Q_3$ form a $6$-MARS. Moreover, note that any vertex of $Q_3$ is a $1$-MARS.

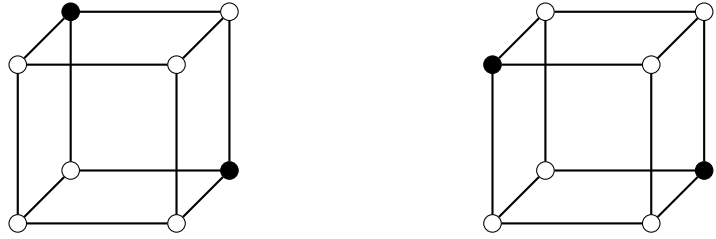
\begin{figure}[ht]
\centering
\begin{tikzpicture}[scale=0.7, transform shape]
\node [draw, shape=circle] (000) at  (0,0) {};
\node [draw, shape=circle] (001) at  (0,3) {};
\node [draw, shape=circle] (010) at  (3,0) {};
\node [draw, shape=circle] (011) at  (3,3) {};

\node [draw, shape=circle] (100) at  (1,1) {};
\node [draw, shape=circle,fill=black] (101) at  (1,4) {};
\node [draw, shape=circle,fill=black] (110) at  (4,1) {};
\node [draw, shape=circle] (111) at  (4,4) {};

\draw[thick](000)--(001)--(011)--(010)--(000);
\draw[thick](100)--(101)--(111)--(110)--(100);
\draw[thick](000)--(100);
\draw[thick](001)--(101);
\draw[thick](011)--(111);
\draw[thick](010)--(110);
\end{tikzpicture}
\hspace{3cm}
\begin{tikzpicture}[scale=0.7, transform shape]
\node [draw, shape=circle] (000) at  (0,0) {};
\node [draw, shape=circle,fill=black] (001) at  (0,3) {};
\node [draw, shape=circle] (010) at  (3,0) {};
\node [draw, shape=circle] (011) at  (3,3) {};

\node [draw, shape=circle] (100) at  (1,1) {};
\node [draw, shape=circle] (101) at  (1,4) {};
\node [draw, shape=circle,fill=black] (110) at  (4,1) {};
\node [draw, shape=circle] (111) at  (4,4) {};

\draw[thick](000)--(001)--(011)--(010)--(000);
\draw[thick](100)--(101)--(111)--(110)--(100);
\draw[thick](000)--(100);
\draw[thick](001)--(101);
\draw[thick](011)--(111);
\draw[thick](010)--(110);
\end{tikzpicture}
\caption{Bold vertices form a $2$-MARS and a $6$-MARS, respectively.}\label{fig:Q_3}
\end{figure}

It stands to reason that $k$-MARS may not exist for any integer $k$ in a graph $G$. Consequently, from this point forward, let $\kappa(G)$ denote the largest integer $k$ such that $G$ contains a $k$-MARS. The basic bounds for $\kappa(G)$ for a graph of order $n$ are the following ones.

\begin{proposition}{\rm \cite[Proposition 1]{estrada-moreno-2026+}}
\label{prop:trivial-bounds}
If $G$ is a graph, then $1\le \kappa(G)\le n(G)-1$. Moreover,
$\kappa(G)= n(G)-1$ if and only if $G$ has a vertex of degree $n(G)-1$.
\end{proposition}

\begin{theorem}{\rm \cite[Theorem 1]{estrada-moreno-2026+}}
For any tree $T$ with at least three vertices, $\kappa(T)\ge 2$.
\end{theorem}

\begin{proposition}{\rm \cite[Proposition 2]{estrada-moreno-2026+}}
The following statements holds for any integers $r$ and $t$ with $r\ge t$ and $n\ge 2$.
\begin{itemize}
  \item[{\rm (i)}] If $t=1$, then $\kappa(K_{r,t})=r+t-1$, and otherwise $\kappa(K_{r,t})=r+t-2$.
  \item[{\rm (ii)}] $\kappa(P_n)=2$.
\end{itemize}
\end{proposition}

The concepts above were introduced recently in~\cite{estrada-moreno-2026+} and is the only work on this topic to date. There, the $k$-multiset antidimension of various graphs families, such as bipartite graphs, trees, and wheels, was studied. We start with the cases of a path and a complete bipartite graph from the mentioned work, and proceed further the rest of combinatorial results concerning the $k$-multiset antidimension of graphs that are known so far.

\begin{remark}{\rm \cite[Remark 2]{estrada-moreno-2026+}}
If $n\ge 2$, then $\adimms_1(P_n)=1$. Moreover, 
$\adimms_2(P_n)=\left\{\begin{array}{ll}
                             1; & n \mbox{ odd}\,, \\
                             2; & n \mbox{ even}\,.
                             \end{array}
  \right.$
\end{remark}

\begin{proposition}{\rm \cite[Proposition 3]{estrada-moreno-2026+}}
Let $r,t$ be two positive integers with $r\ge t$.
\begin{itemize}
\item[{\rm (i)}] If $1< k< t$, then $\adimms_k(K_{r,t})=t-k$.
\item[{\rm (ii)}] If $k=t$, then 
$\adimms_k(K_{r,t})=\left\{\begin{array}{ll}
                                         1; & r>t\,, \\
                                         r; & r=t\,.
                                       \end{array}\right.
$
\item[{\rm (iii)}] If $t< k\le r$, then $\adimms_k(K_{r,t})=r+t-k$.
\item[{\rm (iv)}] If $($$r< k\le r+t-2$ is even and $r+t$ is even$)$ or $($$r< k\le r+t-2$ is odd and $r+t$ is odd$)$, then $\adimms_k(K_{r,t})=r+t-k$.
\item[{\rm (v)}] If $($$r< k\le r+t-2$ is even and $r+t$ is odd$)$ or $($$r< k\le r+t-2$ is odd and $r+t$ is even$)$, then $K_{r,t}$ does not contain any $k$-MARS.
\end{itemize}
\end{proposition}

Notice that any tree $T$ satisfies $2\le \kappa(T)\le |V(T)|-1$. Moreover, both bounds are achieved, paths $P_n$ satisfy $\kappa(T)=2$ and stars $K_{1,t}$ satisfy $\kappa(K_{1,t})=t$. Some partial results are known for trees and are recalled in the following. 

\begin{proposition}{\rm \cite[Corollary 2]{estrada-moreno-2026+}}
For any tree $T$ of diameter at least three, $\adimms_2(T)\in \{1, 2\}$.
\end{proposition}

\begin{proposition}{\rm \cite[Proposition 4]{estrada-moreno-2026+}}
Let $T$ be a tree. Then $\adimms_2(T)=1$ if and only if there is a vertex $x\in V(T)$ such that $|D_j(x)|=2$ for some $j\in [\epsilon(x)]$ and also $|D_i(x)|\ge 2$ for every $i\in [\epsilon(x)]\setminus\{j\}$, where $D_i(x)$ represents the set of vertices at distance $i$ from $x$. 
\end{proposition}

Determining the value of $\adimms_k(T)$ for any tree $T$ when $k\ge 3$ is a challenging problem. The results are known for the case of complete binary trees $T_d$ of height $d$. 
Since $T_1\cong P_3$, we recall the results in cases where $d\ge 2$.

\begin{proposition}{\rm \cite[Proposition 5]{estrada-moreno-2026+}}
If $d\ge 2$, then for any $k\in \{2,2^2,\dots 2^d\}$, 
$$\adimms_k(T_{d})\le \sum_{i=0}^{\log_2 k - 1} 2^{i}\,.$$
\end{proposition}

\begin{proposition}{\rm \cite[Proposition 6]{estrada-moreno-2026+}}
If $d\ge 2$, then 
$$\adimms_3(T_{d})=\left\{\begin{array}{ll}
    \infty; & d=2\,, \\
    2; & d=3\,, \\
    1; & d>3. 
\end{array}\right.$$
\end{proposition}

Observe that a wheel $W_n = C_{n-1} + K_1$ has a maximum degree equal to $n-1$, therefore, from Proposition \ref{prop:trivial-bounds} it follows that $\kappa(W_n)=n-1$.

\begin{theorem}{\rm \cite[Theorem 2]{estrada-moreno-2026+}}
Let $W_n$ be a wheel graph of order $n\ge 7$. If $k\equiv 0\pmod 2$, then for any $4\le k< n-1$,
$$\adimms_k(W_n)=\left\{\begin{array}{ll}
   \dfrac{k+2}{2}; & n\ge\dfrac{5k+2}{2}\,, \\
    n-k; & \dfrac{3k+2}{2}\le n<\dfrac{5k+2}{2}\,, \\
    \infty; & \text{otherwise}\,. 
\end{array}\right.$$
If $k\equiv 1 \pmod 2$, then for any $5\le k< n-1$,
$$\adimms_k(W_n)=\left\{\begin{array}{ll}
   \min\left\{\left\lfloor\dfrac{n-1-k}{3}\right\rfloor+t+1,\dfrac{3k+3}{2}\right\}; & n\ge\dfrac{5k+5}{2}\,, \\
    n-k; & 2k< n<\dfrac{5k+5}{2}\,, \\
    \infty; & \text{otherwise}\,,
\end{array}\right.$$
where $t$ is the remainder of dividing $(n-1-k)$ by $3$. Moreover, $\adimms_{n-1}(W_n) = 1$.
\end{theorem}

In addition to the combinatorial results above, two different linear programming formulations of the $k$-multiset antiresolving set problem were presented in~\cite{estrada-moreno-2026+}. This allowed the authors to calculate the value of the $k$-multiset antidimension of some social graph classes, including some random graphs. These formulations were implemented in a computer assisted system and some experimentation developed. The computations also allowed to evaluate the resistance of such social graphs against active attacks. The interested reader may check~\cite{estrada-moreno-2026+} for details. 

To close this section, we want to recall that the parameter $k$-multiset antidimension was  originated due to its relationship with the area of privacy in social networks, specifically with the concept of $k$-anonymity within a network. Moreover, it is also seen as a multiset version of the $k$-metric antidimension (see \cite{trujillo-2016}), which is also related to privacy in social networks. This multiset variant was introduced in order to better capture the notion of anonymity in a network, based on the fact that, indeed, in a network there is not a real notion of a direction and position. Thus, a multiset makes more real the setting regarding the privacy of a user in a network with respect to active attacks.

In order to better catch an idea on this parameter, we must mention first the classical notion of $k$-anonymity in graphs. That is, a social graph achieves $k$-anonymity with respect to active attacks, if it has probability at least $1/k$ to be identified in the network (see \cite{Samarati-1998}). 
In this sense, in \cite{trujillo-2016} was introduced the concept of $(k, \ell)$-anonymity, which divides the set of users into anonymity sets and claims that users within the same anonymity set are indistinguishable. It is defined that two users are in the same anonymity set if their vectors of distances from the set of attacker nodes to them are equal. Hence, instead of considering a vector of distances, the authors in \cite{estrada-moreno-2026+} proposed considering a multiset of distances, which leads to a new measure called $(k, \ell)$-multiset anonymity. Since it is not the goal of this survey to continue digging into this privacy application, for extra information of this anonymity measure and related issues, we again suggest the article \cite{estrada-moreno-2026+}.

\section{Open problems}
\label{sec:conclude}

In this section we collect a list of open problems. All those problems posed earlier in the literature are accordingly labeled. 

\subsection{Problems on multiset dimension}

\begin{problem}{\rm \cite{hakanen-2024}}
Determine $\mdim(P_n\stp P_m)$, $n,m\ge 3$.
\end{problem}

\begin{problem}{\rm \cite[Conjecture 1]{hfidh-2019}}
Is it true that for any graph $G$ with $\mdim(G)<\infty$ it holds  $\mdim(G)\le n(G)-1$?
\end{problem}

\begin{problem}{\rm \cite[Conjecture 2]{hfidh-2019}}
Let $T$ be a tree of diameter $d$. Is it true that if $\mdim(T)<\infty$, then $\mdim(T)\le n(T)-d+1$?
\end{problem}

\begin{problem} {\rm \cite[Conclusion]{eide-2026+}}
Determine whether $\mdim({\cal G}(n, p))$ is finite w.h.p.\ for $(n-1)p = n^x$, $\frac{1}{8}  < x \le \frac{1}{2}$. 
\end{problem}

\begin{problem} {\rm \cite[Conclusion]{eide-2026+}}
Tighten the gap between the lower and upper bounds in Theorem~\ref{thm:pralat} for $\mdim({\cal G}(n, p))$ for $x \in (0, \frac{1}{2})$.
\end{problem}

\begin{problem}
The idea of finding the set of integers $r$ for which a given graph has a multiset resolving set seems to be natural based on the fact that such structures are not monotonic with respect to the inclusion property. That is, studying the identification spectrum as defined in \cite{Chartrand-2021}. There are only a few results in this direction, and it might be a promising research line.
\end{problem}

As already mentioned, since $Q_n$ is a bipartite graph, $\ldim(Q_n)=1$ for any integer $n\ge 2$. With respect to the multiset and outer multiset dimensions, by computer search, we have obtained the results appearing in Table \ref{tab:Q_n}. These computations indicate that finding the value of such parameters for any hypercube is of clear interest. With respect to the edge multiset dimension of hypercubes, again by computer search, we obtained that $\medim(Q_3)=\medim(Q_4)=\infty$. Thus, a question concerning this is pointed out.

\begin{table}[ht]
    \centering
    \begin{tabular}{|c|c|c||c|c|c|}
    \hline
       $n$  & $\mdim(Q_n)$ & $\odim(Q_n)$ & $n$  & $\mdim(Q_n)$ & $\odim(Q_n)$\\ \hline
       $3$ & $\infty$ & $5$ & $5$ & $7$ & $7$\\ \hline
       $4$ & $8$ & $7$ & $6$ & $8$ & $8$\\ 
    \hline
    \end{tabular}
    \caption{Multiset and outer multiset dimension of small hypercubes.}
    \label{tab:Q_n}
\end{table}

\begin{problem}
Determine the multiset, outer multiset, and edge multiset dimensions of hypercubes. In particular, is it true that $\medim(Q_d)=\infty$ for any $d\ge 3$?
\end{problem}

\subsection{Problems on outer multiset dimension}

\begin{problem}
Determine the computational complexity of {\sc Outer Multiset Dimension Problem} restricted to trees.
\end{problem}

\begin{problem} {\rm \cite[Problem 6.2]{klavzar-2023}}
Investigate the outer multiset dimension of general lexicographic products.
\end{problem}

\begin{problem} {\rm \cite[Problem 6.3]{klavzar-2023}}
Extend Theorem~\ref{thm:grid} to Cartesian products of more that two paths. 
\end{problem}

\begin{problem} {\rm \cite[Problem 6.4]{klavzar-2023}}
\label{prob:odim-torus}
Determine $\odim(C_s\cp C_t)$ for $s,t\ge 3$.
\end{problem}

Assuming Problem~\ref{prob:odim-torus} is solved, we can further ask to extend it to Cartesian products of more that two cycles. 

\begin{problem} {\rm \cite[Problem 3]{pervaiz-2025}}
For an arbitrary graph $G$, determine $\odim(G + K_1)$.
\end{problem}

\begin{problem} {\rm \cite[Problem 4]{pervaiz-2025}}
For an arbitrary graph $G$ and every $m\ge 2$, determine $\odim(G + \overline{K_m})$.
\end{problem}

Since any regular graph $G$ of diameter two satisfies $\odim(G)=n(G)-1$, the following question is worth of considering.

\begin{problem}
Find the outer multiset dimension of every non-regular graph of diameter two.
\end{problem}

\subsection{Problems on local multiset dimension}

\begin{problem}
Characterize the graphs $G$ with $\ldim(G) = \infty$.
\end{problem}

\begin{problem}
    Study the graphs $G$ with $\ldim(G) = 2$.
\end{problem}

\begin{problem}{\rm \cite[Problem 3.1]{alfarisi-2023b}}
    Determine the local multiset dimension of bicyclic graphs in which the two cycles have at least one vertex in common.
\end{problem}

\begin{problem}{\rm \cite[Problem 2]{alfarisi-2023a}}
    Characterize the graphs $G$ with $\ldim(G) = n-1$ or $n-2$.
\end{problem}

\begin{problem}{\rm \cite[Problem 3]{alfarisi-2023a}}
    Characterize the graphs $G$ with $\ldim(G) = {\rm ldim}(G)$.
\end{problem}

\subsection{Problems on edge multiset dimension}

\begin{problem}
Characterize the graphs $G$ with $\medim(G) = \infty$.
\end{problem}

\begin{problem}{\em \cite[Problem 2]{Ikhlaq-2023}}
Characterize the graphs $G$ for which either $\mdim(G)<\medim(G)$, or $\mdim(G)=\medim(G)$, or $\mdim(G)>\medim(G)$. 
\end{problem}

A related version of the problem above can be also as follows. 

\begin{problem}
Show that for any integer $k$, there is graph $G$ for which $\mdim(G)-\medim(G)=k$.
\end{problem}

\begin{problem}{\em \cite[Problem 4]{Ikhlaq-2023}}
If $\medim(G)$ is finite, then can any upper and/or lower bounds with respect to $n(G)$ be found for $\medim(G)$?
\end{problem}

Observing the values of the edge multiset dimension of the subdivision graph $K_n^1$ appearing in the first column from Table \ref{tab:subdiv-edge} (we have computed that also $\medim(K_6^1)=\medim(K_7^1)=\infty$), the following open question seems to be natural to consider. 

\begin{problem}
Is it true that $\medim(K_n^1)=\infty$ for every integer $n\ge 3$?
\end{problem}

Moreover, the values computed for the multiset and edge multiset dimension of the subdivision graph $K_n^k$ suggest to study these two parameters for this family of graphs.


\section*{Acknowledgement}

M.\ Farhan, D.\ Kuziak and I.\ G.\ Yero have been partially supported by ``Ministerio de Ciencia e Innovaci\'on'' through the grant PID2023-146643NB-I00.
Sandi Klav\v zar was supported by the Slovenian Research and Innovation Agency (ARIS) under the grants P1-0297, N1-0285, N1-0355, N1-0431, J1-70045.

\end{document}